\documentclass[11p]{article}
\usepackage{mathrsfs}
\usepackage[french]{babel}
\usepackage[latin1]{inputenc}

\usepackage{amssymb}
\usepackage{pdfsync}
\usepackage{array}
\usepackage{amsmath}
\usepackage{amsthm}
\usepackage{amsfonts}
\usepackage{amstext}
\usepackage{amsopn}
\usepackage{amsbsy}
\usepackage{shortvrb,fancybox}
\usepackage{rotating}
\usepackage{xy}
\xyoption{all}
\usepackage{multirow}

\newcommand{\Z}{\mathbb{Z}}
\newcommand{\Q}{\mathbb{Q}}
\newcommand{\aQ}{\overline{\mathbb{Q}}}

\newcommand{\C}{\mathbb{C}}

\newcommand{\Ax}{\mathbb{A}^{\times}}
\newcommand{\A}{\mathbb{A}}

\newcommand{\Kl}{K^{\lambda}}
\newcommand{\Km}{K^{\mu}}

\newcommand{\DK}{\Delta_{K}}
\newcommand{\OK}{\mathcal{O}_K}
\newcommand{\OL}{\mathcal{O}_L}

\newcommand{\q}{\mathfrak{q}}
\newcommand{\Pcq}{\mathcal{P}^{\mathfrak{q}}}
\newcommand{\OLq}{\mathcal{O}_{L^{\mathfrak{q}}}}
\newcommand{\Hq}{H^{\mathfrak{q}}}

\newcommand{\gq}{\gamma_{\mathfrak{q}}}
\newcommand{\alq}{\alpha_{\mathfrak{q}}}

\newcommand{\Kq}{K_{\mathfrak{q}}}

\newcommand{\kq}{k_{\mathfrak{q}}}

\newcommand{\Iq}{I_{\mathfrak{q}}}

\newcommand{\eq}{e_{\mathfrak{q}}}

\newcommand{\sq}{\sigma_{\mathfrak{q}}}
\newcommand{\sqpr}{\sigma_{\mathfrak{q}'}}

\newcommand{\sqb}{\overline{\sigma}_{\mathfrak{q}}}

\newcommand{\Lq}{L^{\mathfrak{q}}}
\newcommand{\bq}{\beta_{\mathfrak{q}}}
\newcommand{\bqpr}{\beta_{\mathfrak{q}'}}
\newcommand{\bqb}{\overline{\beta_{\mathfrak{q}}}}
\newcommand{\bqprb}{\overline{\beta_{\mathfrak{q}'}}}
\newcommand{\Pq}{P_{\mathfrak{q}}}
\newcommand{\Tq}{T_{\mathfrak{q}}}

\newcommand{\phiEp}{\varphi_{E,p}}

\newcommand{\si}{\sigma}
\newcommand{\Frob}{\mathrm{Frob}}
\newcommand{\h}{\mathrm{H}}

\newcommand{\p}{\mathfrak{p}}
\newcommand{\po}{\mathfrak{p}_0}
\newcommand{\poq}{\mathfrak{p}^{\mathfrak{q}}_0}

\newcommand{\Poq}{\mathcal{P}^{\mathfrak{q}}_0}
\newcommand{\Poqpr}{\mathcal{P}^{\mathfrak{q}'}_0}

\newcommand{\pL}{\mathfrak{p}_L}

\newcommand{\cN}{\mathcal{N}}

\newcommand{\Zp}{\mathbb{Z}_{p}}

\newcommand{\Qp}{\mathbb{Q}_{p}}
\newcommand{\Kp}{K_{\mathfrak{p}}}

\newcommand{\Fp}{\mathbb{F}_{p}}
\newcommand{\Fpx}{\mathbb{F}_{p}^{\times}}

\newcommand{\Ip}{I_{\mathfrak{p}}}
\newcommand{\ep}{e_{\mathfrak{p}}}
\newcommand{\rp}{r_{\mathfrak{p}}}

\newcommand{\ap}{a_{\p}}
\newcommand{\ato}{a_{\tau}}

\newcommand{\GL}{\mathrm{GL}}

\newcommand{\Aut}{\mathrm{Aut}}

\newcommand{\Gal}{\mathrm{Gal}}

\newcommand{\al}{\alpha}

\newcommand{\la}{\lambda}
\newcommand{\lad}{\lambda^{12}}
\newcommand{\amu}{\overline{\mu}}

\newcommand{\kip}{\chi_p}
\newcommand{\J}{\mathcal{J}}

\newcommand{\Mod}{\mkern-5mu \mod}

\def\Overline #1#2#3%
{\mkern #1mu
\overline {\mkern-#1mu #3 \mkern-#2mu}
\mkern #2mu }

\newtheorem*{TheoIrr}{Th\'eor\`eme I (Deux crit\`eres d'irr\'eductibilit\'e)}

\newtheorem{lemme}{Lemme}

\newtheorem*{TheoSerre}{Th\'eor\`eme (Serre,~\cite{[Se]})}
\newtheorem*{TheoMazur}{Th\'eor\`eme (Mazur,~\cite{[Maz]})}
\newtheorem*{TheoMomEff}{Th\'eor\`eme II (Version effective du th\'eor\`eme~A de~\cite{[Mom]})}
\newtheorem*{TheoCheboEff}{Th\'eor\`eme (Chebotarev effectif,~\cite{[LMO]})}
\newtheorem*{corointro}{Corollaire}
\newtheorem*{TheoHomRed}{Th\'eor\`eme III (Homoth\'eties dans le cas r\'eductible)}
\newtheorem*{TheoHom}{Th\'eor\`eme III' (Homoth\'eties)}
\newtheorem*{Question}{Question}

\theoremstyle{plain}
\newtheorem{theoreme}{Th\'eor\`eme}[section]
\newtheorem{proposition}[theoreme]{Proposition}

\theoremstyle{remark}
\newtheorem{definition}[theoreme]{D\'efinition}

\newtheorem{remarques}[theoreme]{Remarques}
\newtheorem{notation}[theoreme]{Notation}
\newtheorem{notations}[theoreme]{Notations}

\AtBeginDocument{}

\newtheorem*{notintro}{Notations}
\newtheorem*{rmqintro}{Remarques}
\title{Caract\`ere d'isog\'enie et crit\`eres d'irr\'eductibilit\'e}
\author{Agn\`es David \\
 {\small Laboratoire de math\'ematiques de Versailles } \\
{\small  Universit\'e de Versailles Saint-Quentin-en-Yvelines } \\
{\small 45 avenue des \'Etats-Unis }\\
{\small  78035 Versailles Cedex } \\
{\small  Agnes.David@ens-lyon.org} 
}
\date{} 

\begin{document}

\maketitle

\begin{abstract}
This article deals with the Galois representation attached to the torsion points of an elliptic curve defined over a number field.
We first determine explicit uniform criteria for the irreducibility of this Galois representation for elliptic curves varying in some infinite families, characterised by their reduction type at some fixed places of the base field.
Then, we deduce from these criteria an explicit form for a bound that appear in a theorem of Momose. Finally, we use these results to precise a previous theorem of the author about the homotheties contained in the image of the Galois representation. 
\end{abstract}

\section*{Introduction}


Cet article traite de la repr\'esentation galoisienne associ\'ee aux points de torsion d'une courbe elliptique d\'efinie sur un corps de nombres.
Son objectif est triple~:
 on d\'etermine d'abord des crit\`eres uniformes d'irr\'eductibilit\'e de cette repr\'esentation galoisienne pour des familles infinies de courbes elliptiques, d\'efinies par un type de r\'eduction prescrit en certaines places du corps de base
 (th\'eor\`eme~I ci-dessous)~;
 on donne ensuite une forme explicite pour une borne annonc\'ee dans un th\'eor\`eme de Momose sur les courbes elliptiques possédant sur un corps de nombres une isogénie de degré premier (th\'eor\`eme~A de l'introduction de~\cite{[Mom]}~; th\'eor\`eme~II ci-dessous)~;
enfin, on d\'eduit de ces travaux une borne uniforme pour les homoth\'eties contenues dans l'image de la repr\'esentation galoisienne consid\'er\'ee, lorsqu'elle est r\'eductible, qui pr\'ecise des r\'esultats pr\'ec\'edents de l'auteur (\cite{[Hudig]}~; th\'eor\`emes~III et~III' ci-dessous).

Le cadre pr\'ecis est le suivant.
On fixe un corps de nombres~$K$ et un plongement de~$K$ dans~$\C$ (dans tout le texte, on consid\'erera ainsi~$K$ comme un sous-corps de~$\C$). 
Soit~$\aQ$ la cl\^oture alg\'ebrique de~$\Q$ dans~$\C$~; on note~$G_K$ le groupe de Galois absolu~$\Gal(\aQ / K)$ de~$K$.
\`A partir de la partie~\ref{sec:compat} et pour toute la fin du texte, on suppose que l'extension~$K/\Q$ est galoisienne.
On fixe une courbe elliptique~$E$ d\'efinie sur $K$ et on note~$j(E)$ son invariant~$j$.
Pour toute place finie du corps~$K$, on dit que le \emph{type de r\'eduction semi-stable} de~$E$ en cette place est \emph{multiplicatif}, \emph{bon ordinaire} ou \emph{bon supersingulier} selon si~$E$ a potentiellement mauvaise r\'eduction multiplicative, bonne r\'eduction ordinaire ou bonne r\'eduction supersinguli\`ere en cette place.
On fixe un nombre premier~$p$ sup\'erieur ou \'egal \`a~$5$ et non ramifi\'e dans~$K$.
 On note~$E_p$ l'ensemble des points de~$E$ dans~$\aQ$ qui sont de~$p$-torsion~; 
  c'est un espace vectoriel de dimension~$2$ sur~$\Fp$, sur lequel le groupe de Galois absolu~$G_K$ de~$K$ agit~$\Fp$-lin\'eairement.
On d\'esigne par~$\phiEp$ la repr\'esentation de~$G_K$ ainsi obtenue~; 
elle prend ses valeurs dans le groupe~$\GL(E_p)$ qui, apr\`es choix d'une base pour~$E_p$, est isomorphe \`a~$\GL_2(\Fp)$.
On note enfin~$\kip$ le caract\`ere cyclotomique de~$G_K$ dans~$\Fpx$~; il co\"incide avec le d\'eterminant de la repr\'esentation~$\phiEp$.


Les questions trait\'ees dans cet article trouvent leur origine dans le th\'eor\`eme suivant de Serre selon lequel, 
lorsque la courbe~$E$ n'a pas de multiplication complexe, la repr\'esentation~$\phiEp$ est \og asymptotiquement surjective \fg.

\begin{TheoSerre}
 On suppose que la courbe $E$ n'a pas de multiplication complexe (sur~$\aQ$). Alors il existe une borne~$C(K,E)$, ne d\'ependant que de~$K$ et de~$E$ et telle que pour tout nombre premier~$p$ strictement sup\'erieur \`a~$C(K,E)$, la repr\'esentation~$\phiEp$ est surjective.
\end{TheoSerre}

La question, pos\'ee \'egalement dans~\cite{[Se]}, d'\'eliminer la d\'ependance en la courbe elliptique~$E$ dans la borne~$C(K , E)$, pour obtenir une version uniforme de ce th\'eor\`eme, s'est r\'ev\'el\'ee ardue.
Mazur a n\'eanmoins d\'emontr\'e que lorsque le corps de base est~$\Q$, la repr\'esentation~$\phiEp$ est  \og uniform\'ement asymptotiquement irr\'eductible \fg\  au sens suivant.

\begin{TheoMazur}
 On suppose que le corps~$K$ est \'egal \`a~$\Q$ et que le nombre premier~$p$ n'appartient pas \`a l'ensemble~$\{2, 3, 5, 7, 13, 11, 17, 19, 37, 43, 67, 163 \}$.
 Alors la repr\'esentation~$\phiEp$ est irr\'eductible.
\end{TheoMazur}
Momose a ensuite obtenu un r\'esultat semblable (avec une borne non effective) lorsque le corps de base est un corps quadratique qui n'est pas imaginaire de nombre de classes~$1$ (th\'eor\`eme~B de l'introduction de~\cite{[Mom]}).

Dans la lign\'ee du th\'eor\`eme de Mazur, il est naturel de consid\'erer le probl\`eme suivant (qui figure par exemple dans l'introduction de~\cite{[Bil]}).
\begin{Question}
Trouver un ensemble infini~$\mathcal{E}$ de courbes elliptiques d\'efinies sur~$K$ et une borne~$C(K,\mathcal{E}) $ ne d\'ependant que de~$K$ et de l'ensemble~$\mathcal{E}$ et v\'erifiant~: 
si~$E$ appartient \`a~$\mathcal{E}$ et~$p$ est strictement sup\'erieur \`a~$C(K,\mathcal{E}) $, alors la repr\'esentation~$\phiEp$ est irr\'eductible.
\end{Question} 
La g\'en\'eralisation du th\'eor\`eme de Mazur au corps de base~$K$ consisterait \`a r\'esoudre cette question pour l'ensemble~$\mathcal{E}$ des courbes elliptiques d\'efinies sur~$K$ qui n'ont pas de multiplication complexe (sur~$\aQ$).


Le pr\'esent article apporte une r\'eponse \`a la question ci-dessus, avec des bornes~$C(K,\mathcal{E}) $ explicites, pour deux classes d'ensembles infinis de courbes elliptiques, d\'efinis par un type de r\'eduction semi-stable des courbes prescrit en des places fix\'ees du corps de base.  

\begin{notintro}
On suppose que le corps~$K$ est galoisien sur~$\Q$.

On note~$d_K$ son degr\'e,~$\Delta_K$ son discriminant,~$h_K$ son nombre de classes d'id\'eaux, et~$R_K$ son r\'egulateur.
On note~$C_1(K)$ la borne~$c_3$ de~\cite{[BuGy]}~: elle ne d\'epend que du degr\'e~$d_K$ de~$K$
 (voir les notations~\ref{not:C1} et la d\'efinition~\ref{def:C1(K)} de ce texte pour une expression explicite de~$C_1(K)$). 
 On pose alors
 $$
C_2(K) = \exp \left( 12 d_K C_1(K) R_K \right)
$$
et pour tout entier~$n$
$$
C(K , n) =  \left( n^{12h_K} C_2(K) + n^{6h_K} \right)^{2d_K}.
$$
\end{notintro}

\begin{TheoIrr}
On suppose que le corps~$K$ est galoisien sur~$\Q$. 
\begin{enumerate}
\item Soit~$M$ un entier naturel sup\'erieur ou \'egal~$1$.

On note~$\mathcal{E}(K ; M)$ l'ensemble des courbes elliptiques~$E$ d\'efinies sur~$K$ pour lesquelles il existe
\begin{itemize}
\item[$\bullet$] $q$ et~$q'$ des nombres premiers totalement d\'ecompos\'es dans~$K$ et inf\'erieurs ou \'egaux \`a~$M$,
\item[$\bullet$] une place~$\q$ de~$K$ au-dessus de~$q$ et une place~$\q'$ de~$K$ au-dessus de~$q'$ tels que les types de r\'eduction semi-stable de~$E$ en~$\q$ et~$\q'$ sont diff\'erents
\end{itemize}
 (les premiers~$q$ et~$q'$ et les places~$\q$ et~$\q'$ peuvent d\'ependre de~$E$).

Alors pour toute courbe elliptique~$E$ dans~$\mathcal{E}(K ; M)$ et tout nombre premier~$p$ strictement sup\'erieur \`a~$C(K , M)$, la repr\'esentation~$\varphi_{E,p}$ est irr\'eductible.
\item Soit~$q$ un nombre premier rationnel  totalement d\'ecompos\'e dans~$K$.

On note~$\mathcal{E}'(K ; q)$ l'ensemble des courbes elliptiques~$E$ d\'efinies sur~$K$ v\'erifiant~: il existe une place  (pouvant d\'ependre de~$E$) de~$K$ au-dessus de~$q$ en laquelle~$E$ a potentiellement mauvaise r\'eduction multiplicative.

On pose
$$
B(K ; q) = \max \left( C(K,q) ,  \left(  1 + 3^{6 d_K h_K}  \right)^2 \right).
$$
Alors pour toute courbe elliptique~$E$ dans $\mathcal{E}(K ; q)$ et tout nombre premier~$p$ strictement sup\'erieur \`a~$B(K ; q)$, la repr\'esentation $\varphi_{E,p}$ est irr\'eductible.
\end{enumerate}
\end{TheoIrr}

En utilisant le th\'eor\`eme~I, associ\'e \`a une forme effective du th\'eor\`eme de Chebotarev (\cite{[LMO]}), on donne ensuite une formule explicite pour une borne~$C_K$ satisfaisant le th\'eor\`eme~A de~\cite{[Mom]}, dont on rappelle ici l'\'enonc\'e
($A$ d\'esigne une constante absolue qui appara\^it dans le th\'eor\`eme de Chebotarev effectif de~\cite{[LMO]}, voir partie~\ref{ssec:Cheboeff} de ce texte~; voir aussi~\cite{[Dav08]} pour un \'enonc\'e similaire).

\begin{TheoMomEff}
On suppose que le corps~$K$ est galoisien sur~$\Q$.
On pose
$$
C_K = \max \left(  C\left( K , 2(\DK)^{Ah_K}  \right) ,   \left(  1 + 3^{6 d_K h_K}  \right)^2   \right).
$$
On suppose que~$p$ est strictement sup\'erieur \`a~$C_K$ et que la courbe~$E$ poss\`ede une isog\'enie de degr\'e~$p$ d\'efinie sur~$K$.
On note~$\la$ le caract\`ere de~$G_K$ dans~$\Fpx$ donnant l'action de~$G_K$ sur le sous-groupe d'ordre~$p$ de~$E_p$ d\'efinissant l'isog\'enie.
Alors on est dans l'un des deux cas suivants.
\begin{description}
	\item[Type supersingulier]\ 
	\begin{enumerate}
	\item Le nombre premier~$p$ est congru \`a $3$ modulo $4$~;
	\item En toute place de~$K$ au-dessus de~$p$, la courbe~$E$ a mauvaise r\'eduction additive et potentiellement bonne r\'eduction supersinguli\`ere.
	\item La puissance sixi\`eme du caract\`ere~$\la$ est \'egale \`a $\kip^{3 + \frac{p-1}{2}}$.
	\end{enumerate}
	\item[Type ordinaire] Il existe un corps quadratique imaginaire~$L$ satisfaisant les conditions suivantes.
			\begin{enumerate}
				\item Le corps~$K$ contient~$L$ et son corps de classes de Hilbert (en particulier, la norme dans l'extension~$K/L$ de tout id\'eal fractionnaire de~$K$ est un id\'eal fractionnaire principal de~$L$).
				\item Le nombre premier~$p$ est d\'ecompos\'e dans $L$.
				\item Il existe un id\'eal premier~$\pL$ de~$L$ au-dessus de~$p$ tel que le caract\`ere~$\lad$ est non ramifi\'e aux places finies de~$K$ premi\`eres \`a~$\pL$.
				\item Soient~$\q$ un id\'eal premier de~$K$ premier \`a~$\pL$,~$\alq$ dans~$\OL$ un g\'en\'erateur 
				de l'id\'eal~$N_{K/L}(\q)$ et~$\sq$ dans~$G_K$ un rel\`evement  du frobenius du corps r\'esiduel de~$K$ en~$\q$~;
				alors on a~$\lad(\sq)  = \alq^{12} \Mod \pL $.
			\end{enumerate}
\end{description}
\end{TheoMomEff}

\begin{rmqintro}\ 
\begin{enumerate}
\item L'\'enonc\'e initial du th\'eor\`eme~A de~\cite{[Mom]} pr\'esente un troisi\`eme cas possible, dans lequel le caract\`ere~$\lad$ ou le caract\`ere~$\left( \kip \la^{-1} \right)^{12}$ est trivial.
 On l'a ici \'elimin\'e avec les bornes uniformes pour l'ordre des points de torsion d'une courbe elliptique qui figurent dans~\cite{[Mer96]} et~\cite{[Pa]}.
\item La terminologie \emph{type supersingulier ou ordinaire} est propre au pr\'esent article~;
 elle fait r\'ef\'erence au type de r\'eduction semi-stable de la courbe~$E$, non pas aux places de~$K$ au-dessus de~$p$, mais en toute place d'une famille finie ne d\'ependant que de~$K$
 (voir la d\'efinition~\ref{def:JK} et la proposition~\ref{prop:mmredJK}).
 \item D'apr\`es Momose (remarque~8, p.~341 de~\cite{[Mom]}), l'hypoth\`ese de Riemann g\'en\'eralis\'ee entra\^inerait que le cas supersingulier ne se produit pas, pour~$p$ assez grand.
 Larson et Vaintrob obtiennent ainsi, sous cette hypoth\`ese, un analogue du th\'eor\`eme~II (voir~\cite{[LV]}, \S5).
\end{enumerate} 
\end{rmqintro}

Le th\'eor\`eme~II a pour cons\'equence un crit\`ere d'irr\'eductibilit\'e pour l'ensemble des courbes elliptiques semi-stables sur~$K$, lorsque les propri\'et\'es du corps~$K$ emp\^echent le cas ordinaire de survenir.
Par exemple (voir aussi l'appendice~B de~\cite{[Krau]})~:
\begin{corointro}
On suppose que le corps~$K$ est galoisien sur~$\Q$ et ne contient le  corps de classes de Hilbert d'aucun corps quadratique imaginaire. Alors pour toute courbe elliptique~$E$ semi-stable sur~$K$ et tout nombre premier~$p$ strictement sup\'erieur \`a~$C_K$, la repr\'esentation~$\phiEp$ est irr\'eductible. 
\end{corointro}


Enfin, l'\'etude men\'ee pour de la d\'emonstration du th\'eor\`eme~II donne le r\'esultat uniforme suivant pour les homoth\'eties contenues dans l'image de la repr\'esentation~$\phiEp$.
\begin{TheoHomRed}
On se place dans les hypoth\`eses du th\'eor\`eme~II.
\begin{enumerate}
\item Dans le cas supersingulier, l'image de~$\phiEp$ contient les carr\'es des homoth\'eties.
\item Dans le cas ordinaire, l'image de~$\phiEp$ contient les puissances douzi\`emes des homoth\'eties.
\end{enumerate}
\end{TheoHomRed}
Avec les r\'esultats de~\cite{[Hudig]} lorsque la repr\'esentation est irr\'eductible, on obtient l'\'enonc\'e g\'en\'eral suivant.
\begin{TheoHom}
On suppose que le corps~$K$ est galoisien sur~$\Q$. Alors pour toute courbe elliptique~$E$ d\'efinie sur~$K$ et tout nombre premier~$p$ strictement sup\'erieur \`a~$C_K$, on est dans l'un des deux cas suivants~:
\begin{enumerate}
\item l'image de~$\phiEp$ contient les carr\'es des homoth\'eties~;
\item la repr\'esentation~$\phiEp$ est r\'eductible de type ordinaire~; dans ce cas l'image de~$\phiEp$ contient les puissances douzi\`emes des homoth\'eties.
\end{enumerate}
\end{TheoHom}

Dans toute la suite du texte, \`a l'exception de la partie~\ref{sec:critirr}, on suppose que la repr\'esentation~$\phiEp$ est r\'eductible.
La courbe~$E$ poss\`ede alors un sous-groupe d'ordre~$p$ d\'efini sur~$K$.
On fixe un tel sous-groupe~$W$~; il lui est associ\'e une isog\'enie de~$E$ de degr\'e~$p$, d\'efinie sur~$K$.
L'action de~$G_K$ sur~$W(\aQ)$ est donn\'ee par un caract\`ere continu de~$G_K$ dans $\Fpx$~; on le note $\la$ et, suivant la terminologie introduite dans~\cite{[Maz]}, on l'appelle le caract\`ere d'isog\'enie associ\'e au couple~$(E,W)$.
On fixe \'egalement une base de $E_p$ dont le premier vecteur engendre~$W(\aQ)$~; dans cette base la matrice de la repr\'esentation~$\phiEp$ est triangulaire sup\'erieure, de caract\`eres diagonaux~$(\la, \kip\la^{-1})$.


Pour la d\'etermination de la borne~$C_K$, on suit la m\'ethode introduite dans~\cite{[Mom]}.

Dans la partie~\ref{sec:etloc}, on \'etablit les propri\'et\'es locales du caract\`ere d'isog\'enie~:
 \`a l'aide des r\'esultats de~\cite{[SeTa]},~\cite{[Se]} et~\cite{[Ray]}, on d\'etermine sa restriction aux sous-groupes d'inertie de~$G_K$ et l'image par~$\la$ d'un \'el\'ement de Frobenius associ\'e \`a une place hors de~$p$, en fonction du type de r\'eduction semi-stable de la courbe en cette place.

Dans la partie~\ref{sec:compat}, on utilise la th\'eorie du corps de classes (appliqu\'ee \`a l'extension ab\'elienne trivialisant une puissance du caract\`ere d'isog\'enie) pour relier le type de r\'eduction de la courbe en une place hors de~$p$ \`a l'action sur le sous-groupe d'isog\'enie des sous-groupes d'inertie des places au-dessus de~$p$.
La d\'etermination de la borne au-del\`a de laquelle~$p$ doit \^etre pris pour qu'on puisse \'etablir un tel lien fait intervenir des bornes de Bugeaud et Gy{\H{o}}ry (\cite{[BuGy]}) sur la hauteur d'un repr\'esentant d'une classe d'entiers de~$K$ modulo ses unit\'es.

Dans la partie~\ref{sec:critirr}, on utilise les r\'esultats de la partie~\ref{sec:compat} pour d\'emontrer les deux crit\`eres d'irr\'eductibilit\'e qui constituent le th\'eor\`eme~I ci-dessus~; 
la d\'emonstration de ce th\'eor\`eme n\'ecessite \'egalement les bornes uniformes pour l'ordre des points de torsion des courbes elliptiques qui figurent dans~\cite{[Mer96]} et~\cite{[Pa]}.

Dans la partie~\ref{sec:carisohom}, on emploie les r\'esultats de la partie~\ref{sec:critirr}, associ\'es \`a une version effective du th\'eor\`eme de Chebotarev (\cite{[LMO]}), pour \'etablir la forme de la borne~$C_K$ du th\'eor\`eme~II~;
 on v\'erifie ensuite (partie~\ref{ssec:2cariso}) qu'elle permet d'aboutir aux conclusions  du th\'eor\`eme~II.
 Enfin, on fait le lien entre les deux types du th\'eor\`eme~II et l'\'etude locale initiale du caract\`ere d'isog\'enie (partie~\ref{sec:etloc}) pour obtenir le th\'eor\`eme~III.

\section{\'Etude locale du caract\`ere d'isog\'enie}\label{sec:etloc}

\subsection{D\'efaut de semi-stabilit\'e}\label{ssec:defsemstab}

Soient~$\ell$ un nombre premier rationnel et~$\mathfrak{L}$ un id\'eal premier de~$K$ au-dessus de~$\ell$.
\begin{notations}\label{not:dicr}
On fixe~$D_{\mathfrak{L}}$ un sous-groupe de d\'ecomposition pour~$\mathfrak{L}$ dans~$G_K$~; on note~$I_{\mathfrak{L}}$ son sous-groupe d'inertie (deux choix diff\'erents de~$D_{\mathfrak{L}}$ sont conjugu\'es par un \'el\'ement de~$G_K$~; leur image par le caract\`ere ab\'elien~$\la$ co\" incident donc).
On note~$K_{\mathfrak{L}}$ le compl\'et\'e de~$K$ en~$\mathfrak{L}$,~$\overline{K_{\mathfrak{L}}}$ sa cl\^oture alg\'ebrique associ\'ee au sous-groupe~$D_{\mathfrak{L}}$ et  $K^{nr}_{\mathfrak{L}}$ l'extension non ramifi\'ee maximale de~$K_{\mathfrak{L}}$ dans~$\overline{K_{\mathfrak{L}}}$.
On note $k_{\mathfrak{L}}$ le corps r\'esiduel de~$K$ en~$\mathfrak{L}$,~$N\mathfrak{L}$ son cardinal et~$\overline{k_{\mathfrak{L}}}$ sa cl\^oture alg\'ebrique associ\'ee \`a~$\overline{K_{\mathfrak{L}}}$.
 On fixe dans~$D_{\mathfrak{L}}$ un rel\`evement~$\sigma_{\mathfrak{L}}$ du frobenius de~$k_{\mathcal{L}}$ (le sous-groupe~$D_{\mathfrak{L}}$ \'etant fix\'e, deux tels rel\`evements diff\`rent par un \'el\'ement de~$I_{\mathfrak{L}}$~; leurs images par un caract\`ere non ramifi\'e en~$\mathcal{L}$ co\"incident donc).
 
 Enfin, on note~$\Kl$ l'extension de~$K$ trivialisant le caract\`ere~$\la$~; l'extension~$\Kl / K$ est galoisienne, cyclique, de degr\'e divisant~$p-1$. La courbe~$E$ poss\`ede un point d'ordre~$p$ d\'efini sur~$\Kl$.
\end{notations}

\begin{lemme}\label{lem:nonredadd}
On suppose~$\ell$ diff\'erent de~$p$~; alors en toute place de~$\Kl$ au-dessus de~$\mathfrak{L}$,~$E$ n'a pas r\'eduction additive.
\end{lemme}

\begin{proof} 
On renvoie pour la d\'emonstration au~\S 6 de~\cite{[Maz]} ou au lemme~3.4 (\S3.2.2) de~\cite{[Hudig]}.
\end{proof}

\subsubsection{D\'efinition de~$e_{\mathfrak{L}}$ lorsque $j(E)$ n'est pas entier en~$\mathfrak{L}$}\label{sssec:eLmred}

On suppose que l'invariant~$j$ de~$E$ n'est pas entier en la place~$\mathfrak{L}$, c'est-\`a-dire que~$E$ a potentiellement r\'eduction multiplicative en~$\mathfrak{L}$.

Alors il existe une unique extension~$K'_{\mathfrak{L}}$  de~$K_{\mathfrak{L}}$ de degr\'e inf\'erieur ou \'egal \`a~$2$, sur laquelle~$E$ est isomorphe \`a une courbe de Tate~;
cette extension est de degr\'e~$1$ si et seulement si~$E$ a r\'eduction multiplicative d\'eploy\'ee en~$\mathfrak{L}$, de degr\'e~$2$ sinon~;
elle est ramifi\'ee si et seulement si~$E$ a r\'eduction additive en~$\mathfrak{L}$ (voir~\cite{[Sil]} appendice~C th\'eor\`eme~14.1).

On note~$e_{\mathfrak{L}}$ l'indice de ramification de cette extension~; on a donc $e_{\mathfrak{L}}$ \'egal \`a~$1$ si et seulement si~$E$ a r\'eduction multiplicative (d\'eploy\'ee ou non) en~$\mathfrak{L}$ et $e_{\mathfrak{L}}$ \'egal \`a~$2$ si et seulement si~$E$ a r\'eduction additive en~$\mathfrak{L}$.

\subsubsection{D\'efinition de~$e_{\mathfrak{L}}$ lorsque $j(E)$ est entier en~$\mathfrak{L}$}\label{sssec:eLbred}

On suppose que l'invariant~$j$ de~$E$ est entier en la place~$\mathfrak{L}$, c'est-\`a-dire que~$E$ a potentiellement bonne r\'eduction en~$\mathfrak{L}$.

Alors il existe une plus petite extension de~$K^{nr}_{\mathfrak{L}}$ sur laquelle~$E$ a bonne r\'eduction et cette extension est galoisienne (\cite{[SeTa]},~\S 2, corollaire~3,  p.~498).
On note~$M_{\mathfrak{L}}$ cette extension et~$\Phi_{\mathcal{L}}$ le groupe de Galois de~$M_{\mathfrak{L}}$ sur $K^{nr}_{\mathfrak{L}}$.

Lorsque~$\ell$ est diff\'erent de~$p$, on sait de plus (\cite{[SeTa]}, \textit{loc. cit.}) que~$M_{\mathfrak{L}}$ est l'extension de~$K^{nr}_{\mathfrak{L}}$ engendr\'ee par les coordonn\'ees des points de~$p$-torsion de~$E$.
En particulier, le corps~$M_{\mathfrak{L}}$ contient l'extension de~$K^{nr}_{\mathfrak{L}}$ engendr\'ee par les coordonn\'ees des points du sous-groupe d'isog\'enie~$W$~; on note~$M^{\la}_{\mathfrak{L}}$ cette extension.
D'apr\`es le lemme~\ref{lem:nonredadd},~$E$ a bonne r\'eduction sur~$M^{\la}_{\mathfrak{L}}$~; par minimalit\'e de~$M_{\mathfrak{L}}$, le corps~$M_{\mathfrak{L}}$ est donc inclus dans~$M^{\la}_{\mathfrak{L}}$.
Ainsi, les corps~$M_{\mathfrak{L}}$ et~$M^{\la}_{\mathfrak{L}}$ co\"incident et le groupe~$\Phi_{\mathcal{L}}$ s'identifie au sous-groupe d'inertie en~$\mathfrak{L}$ de l'extension ab\'elienne~$\Kl / K$. 
En particulier, le groupe~$\Phi_{\mathfrak{L}}$ est cyclique et son ordre divise~$p-1$.

Par ailleurs, que~$\ell$ soit \'egal ou diff\'erent de~$p$, le groupe~$\Phi_{\mathcal{L}}$ s'identifie \`a un sous-groupe du groupe des automorphismes de la courbe elliptique d\'efinie sur~$\overline{k_{\mathfrak{L}}}$
 qu'on obtient par r\'eduction de~$E \times_K M_{\mathfrak{L}}$ (\cite{[SeTa]},~\S2, d\' emonstration du th\'eor\`eme~2, p.~497).
On note~$\widetilde{E}_{\mathfrak{L}}$ cette courbe elliptique r\'eduite et~$\Aut ( \widetilde{E}_{\mathfrak{L}} ) $  son groupe d'automorphismes.
Le groupe~$\Aut ( \widetilde{E}_{\mathfrak{L}} ) $ d\'epend de l'invariant $j$ de~$\widetilde{E}_{\mathfrak{L}}$, qui est \'egal \`a la classe de~$j(E)$ modulo~$\mathfrak{L}$, de la mani\`ere suivante
 (\cite{[Sil]} appendice~A, proposition~1.2 et exercice~A.1)~:
 \begin{itemize}
 	\item[$\bullet$] si~$j(E)$ est diff\'erent de~$0$ et~$1728$ modulo~$\mathfrak{L}$,~$\Aut ( \widetilde{E}_{\mathfrak{L}} ) $ est cyclique d'ordre~$2$~;
 	\item[$\bullet$] si~$\ell$ est diff\'erent de~$2$ et de~$3$ et~$j(E)$ est congru \`a~$1728$ modulo~$\mathfrak{L}$,~$\Aut ( \widetilde{E}_{\mathfrak{L}} ) $ est cyclique d'ordre $4$~; 
		 \item[$\bullet$] si~$\ell$ est diff\'erent de~$2$ et de~$3$ et~$j(E)$ est congru \`a~$0$ modulo~$\mathfrak{L}$,~$\Aut ( \widetilde{E}_{\mathfrak{L}} ) $ est cyclique d'ordre~$6$~; 
		 	 \item[$\bullet$] si~$\ell$ est \'egal \`a~$3$ et~$j(E)$ est congru \`a~$0 = 1728$ modulo~$\mathfrak{L}$,~$\Aut ( \widetilde{E}_{\mathfrak{L}} ) $ est un groupe d'ordre~$12$, produit semi-direct d'un groupe cyclique 	d'ordre~$3$ par un groupe cyclique d'ordre~$4$ (le deuxi\`eme agissant sur le premier de l'unique mani\`ere non triviale)~; on v\'erifie que les sous-groupes cycliques d'un tel groupe sont d'ordre~$1$,~$2$,~$3$,~$4$ ou~$6$~;
 	\item[$\bullet$] si~$\ell$ est \'egal \`a~$2$ et~$j(E)$ est congru \`a~$0 = 1728$ modulo~$\mathfrak{L}$,~$\Aut ( \widetilde{E}_{\mathfrak{L}} ) $ est un groupe d'ordre~$24$, produit semi-direct du groupe des quaternions (d'ordre~$8$) par un groupe cyclique d'ordre~$3$ (le deuxi\`eme agissant sur le premier  en permutant les g\'en\'erateurs)~; on v\'erifie que les sous-groupes cycliques d'un tel groupe sont \'egalement d'ordre~$1$,~$2$,~$3$,~$4$ ou~$6$.
\end{itemize}

On d\'eduit de ce qui pr\'ec\`ede que pour tout~$\ell$, le groupe~$\Phi_{\mathfrak{L}}$ est cyclique d'ordre~$1$,~$2$,~$3$,~$4$ ou~$6$.
On note~$e_{\mathfrak{L}}$ l'ordre de~$\Phi_{\mathcal{L}}$~; il est donc dans l'ensemble~$\left\{1, 2 , 3 , 4, 6 \right\}$ et v\'erifie~:
\begin{itemize}
\item[$\bullet$] $e_{\mathfrak{L}}$ est \'egal \`a~$1$ si et seulement si~$E$ a bonne r\'eduction en~$\mathfrak{L}$~;
\item[$\bullet$] si~$e_{\mathfrak{L}}$ est \'egal \`a~$4$, alors~$j(E)$ est congru \`a $ 1728$ modulo $\mathfrak{L}$~;
\item[$\bullet$] si~$e_{\mathfrak{L}}$ est \'egal \`a~$3$ ou~$6$, alors~$j(E)$ est congru \`a $0$ modulo $\mathfrak{L}$~;
\item[$\bullet$] si~$\ell$ est diff\'erent de~$p$, alors $e_{\mathfrak{L}}$ est l'ordre de~$\la(I_{\mathfrak{L}})$ et divise~$p-1$.
\end{itemize}

\subsection{Action des sous-groupes d'inertie des places au-dessus de $p$}\label{ssec:inenp}

Soit~$\p$ un id\'eal premier de~$K$ situ\'e au-dessus de~$p$~; on reprend les notations~\ref{not:dicr}. Cette partie pr\'ecise la proposition~3.2 (\S~3.1) de~\cite{[Hudig]} qui d\'ecrit la restriction de la puissance douzi\`eme du caract\`ere d'isog\'enie \`a un sous-groupe d'inertie~$\Ip$ associ\'e \`a~$\p$.

 \begin{proposition}\label{prop:inpmult}
On suppose que~$E$ a potentiellement r\'eduction multiplicative en~$\p$. Alors
\begin{enumerate}
\item le caract\`ere~$\la^{\ep}$ restreint \`a~$\Ip$ est trivial ou \'egal \`a~$\kip^{\ep}$~;
\item en particulier $\la^{2}$ restreint \`a~$\Ip$ est trivial ou \'egal \`a~$\kip^{2}$.
\end{enumerate}
Lorsque~$p$ est sup\'erieur ou \'egal \`a~$17$, il existe un unique entier~$\ap$ valant~$0$ ou~$12$ tel que le caract\`ere~$\la^{12}$ restreint \`a~$\Ip$ co\"incide avec~$\kip^{\ap}$.
\end{proposition}

\begin{proof}

Soit~$\Kp'$ l'unique extension de~$\Kp$ de degr\'e inf\'erieur ou \'egal \`a~$2$ sur laquelle~$E$ est isomorphe \`a une courbe de Tate (voir partie~\ref{sssec:eLmred})~; 
 on note~$\Ip'$ le sous-groupe d'inertie de~$D_{\p}$ associ\'e \`a~$\Kp'$.
 
 D'apr\`es~\cite{[Se]} (proposition~13 de \S1.12 et page~273 de \S1.11), le caract\`ere~$\la$ restreint au sous-groupe~$\Ip'$ est soit trivial soit \'egal au caract\`ere cyclotomique~$\kip$. Par d\'efinition de~$\ep$ (partie~\ref{sssec:eLmred}),~$\Ip'$ est un sous-groupe d'indice~$\ep$ de~$\Ip$~; on en d\'eduit le premier point de la proposition~; le second d\'ecoule du fait que~$\ep$ est \'egal \`a~$1$ ou~$2$.
\end{proof}

\begin{proposition}\label{prop:inpbon}
On suppose que~$E$ a potentiellement bonne r\'eduction en~$\p$. Alors il existe un entier~$\rp$ compris entre~$0$ et~$\ep$
tel que le caract\`ere~$\la^{\ep}$ restreint au sous-groupe d'inertie~$\Ip$ co\"incide avec~$\kip^{r_{\p}}$
(les couples~$(\ep , \rp)$ possibles, ainsi que des informations suppl\'ementaires pour certains cas, sont rassembl\'es dans le tableau suivant).
En particulier, lorsque~$p$ est sup\'erieur ou \'egal \`a~$17$,  il existe un unique entier~$\ap$ dans l'ensemble~$\{0 , 4 , 6 , 8 , 12\}$ tel que le caract\`ere~$\la^{12}$ restreint \`a~$\Ip$ co\"incide avec~$\kip^{\ap}$.
\renewcommand\arraystretch{1.1}
$$
\begin{array}{| c | c | c | c | c | c |}
\hline
\multirow{2}{*}{$\ep$} & \multirow{2}{*}{$\rp$} & \multirow{2}{*}{$\ap = \displaystyle{\frac{12}{\ep}\rp}$} & \multirow{2}{*}{$p$} & \multirow{2}{*}{$j(E)$} &  \text{Type de r\'eduction} \\
 &  &  &  &  &  \text{semi-stable en } \p \\
\hline
\multirow{2}{*}{$1$} & 0 & 0 & \multirow{2}{*}{$-$}  & \multirow{2}{*}{$-$} & \multirow{2}{*}{$-$} \\  \cline{2-3} 
 & 1 & 12 &  &  &  \\ \hline \hline
 \multirow{2}{*}{$2$} & 0 & 0 &  \multirow{2}{*}{$-$} &  \multirow{2}{*}{$-$}  & \multirow{2}{*}{$-$} \\  \cline{2-3} 
  & 2 & 12 &  & &  \\ \hline \hline
 \multirow{4}{*}{$3$} & 0 & 0 & p \equiv 1 \Mod 3 & \multirow{4}{*}{$j(E) \equiv 0 \Mod \p$} & \text{ordinaire} \\  \cline{2-4} \cline{6-6} 
 & 1 & 4 & \multirow{2}{*}{$p \equiv 2 \Mod 3 $} &   & \multirow{2}{*}{supersingulier} \\  \cline{2-3}   
 & 2 & 8 &  &  & \\ \cline{2-4} \cline{6-6}  
 & 3 & 12 & p \equiv 1 \Mod 3 & & \text{ordinaire} \\ \hline \hline
 \multirow{3}{*}{$4$} & 0 & 0 & p \equiv 1 \Mod 4 &  \multirow{3}{*}{$j(E) \equiv 1728 \Mod \p$} & \text{ordinaire} \\ \cline{2-4} \cline{6-6}  
 & 2 & 6 & p \equiv 3 \Mod 4  &   & \text{supersingulier} \\ \cline{2-4} \cline{6-6}  
 & 4 & 12 & p \equiv 1 \Mod 4 &  & \text{ordinaire} \\ \hline \hline
 \multirow{4}{*}{$6$} & 0 & 0 & p \equiv 1 \Mod 3 & \multirow{4}{*}{$j(E) \equiv 0 \Mod \p$} & \text{ordinaire} \\ \cline{2-4} \cline{6-6}  
  & 2 & 4 & \multirow{2}{*}{$p \equiv 2 \Mod 3 $}&  & \multirow{2}{*}{supersingulier} \\ \cline{2-3} 
  & 4 & 8 &  &  &  \\ \cline{2-4} \cline{6-6} 
 & 6 & 12 & p \equiv 1 \Mod 3 &  & \text{ordinaire} \\ \hline 
\end{array}
$$

\end{proposition}

\begin{proof}
Pour l'existence et les valeurs possibles de~$\rp$ et~$\ap$ et les trois premi\`eres colonnes du tableau, on renvoie \`a la d\'emonstration de la proposition~3.2 (\S~3.1) de~\cite{[Hudig]} (voir aussi la remarque~1 de la partie~2 de~\cite{[Mom]}).

Pour la quatri\`eme colonne, la d\'emonstration de la proposition~3.2 de~\cite{[Hudig]} donne qu'il existe un entier~$\ap'$ satisfaisant la congruence~$\ep\ap' \equiv \rp \mod p-1$.
Alors~:
\begin{itemize}
\item[$\bullet$] si~$3$ divise~$\ep$ et~$\rp$ (couples~$(3,0)$,~$(3,3)$,~$(6,0)$ et~$(6,6)$), alors~$3$ divise~$p-1$, donc~$p$ est congru \`a~$1$ modulo~$3$~;
\item[$\bullet$] si~$3$ divise~$\ep$ et ne divise pas~$\rp$ (couples~$(3,1)$,~$(3,2)$,~$(6,2)$ et~$(6,4)$), alors~$3$ ne divise pas~$p-1$, donc~$p$ est congru \`a~$2$ modulo~$3$~;
\item[$\bullet$] si~$\ep$ est \'egal \`a~$4$ et divise~$\rp$ (couples~$(4,0)$ et~$(4,4)$), alors~$4$ divise~$p-1$ donc~$p$ est congru \`a~$1$ modulo~$4$~;
\item[$\bullet$] si~$\ep$ est \'egal \`a~$4$ et ne divise pas~$\rp$ (couple~$(4,2)$), alors~$4$ ne divise pas~$p-1$ donc~$p$ est congru \`a~$3$ modulo~$4$.
\end{itemize}

Pour la cinqui\`eme colonne, la discussion de la partie~\ref{sssec:eLbred} donne~:
\begin{itemize}
\item[$\bullet$] si~$\ep$ est \'egal \`a~$4$, alors~$j(E)$ est congru \`a~$1728$ modulo~$\p$~;
\item[$\bullet$] si~$3$ divise~$\ep$, alors~$j(E)$ est congru \`a~$0$ modulo~$\p$.
\end{itemize}

Enfin, la derni\`ere colonne r\'esulte de la d\'etermination de l'invariant de Hasse des courbes d'invariant~$j$ \'egal~$0$ ou~$1728$ sur un corps fini de caract\'eristique~$p$ (sup\'erieur ou \'egal \`a~$5$~;
voir~\cite{[Sil]},~\S V.4, exemples~4.4 et~4.5).
En effet, soit~$k$ un corps fini de caract\'eristique~$p$~; alors~:
\begin{itemize}
\item[$\bullet$] la courbe elliptique d\'efinie sur~$k$ par l'\'equation~$y^2 = x^3 + 1$, d'invariant~$j$ \'egal \`a~$0$, est ordinaire si et seulement si~$p$ est congru \`a~$1$ modulo~$3$~;
\item[$\bullet$] la courbe elliptique d\'efinie sur~$k$ par l'\'equation~$y^2 = x^3 + x$, d'invariant~$j$ \'egal \`a~$1728$, est ordinaire si et seulement si~$p$ est congru \`a~$1$ modulo~$4$.
\end{itemize}
\end{proof}

\subsection{Ramification et action du frobenius aux places hors de~$p$}\label{ssec:ramfrobhorsp}

Soient~$q$ un nombre premier rationnel diff\'erent de~$p$ et~$\q$ un id\'eal premier de~$K$ au-dessus de~$q$.

\subsubsection{Lorsque~$j(E)$ n'est pas entier en~$\q$}
\begin{proposition}\label{prop:ramfrobhorspmult}
On suppose que~$E$ a potentiellement r\'eduction multiplicative en~$\q$.
Alors~:
\begin{enumerate}
\item le groupe~$\la(\Iq)$ est d'ordre~$\eq$~; en particulier, le caract\`ere~$\lad$ est non ramifi\'e en~$\q$~;
\item $\la^2(\sq)$ vaut $1$ ou $(N\q)^2$ modulo $p$.
\end{enumerate}
\end{proposition}

\begin{proof}
La proposition~3.3 (\S 3.2.1) de~\cite{[Hudig]} donne la deuxi\`eme assertion et que~$\la(\Iq)$ est d'ordre au plus~$2$.
Comme~$\eq$ est \'egal \`a~$1$ ou~$2$ (voir partie~\ref{sssec:eLmred}), il ne reste qu'\`a prouver qu'on ne peut avoir \`a la fois~$\eq$ \'egal \`a~$2$ et~$\la(\Iq)$ trivial.

 Supposons par l'absurde que c'est le cas.
 Alors, avec les notations de la partie~\ref{ssec:defsemstab},~$E$ a mauvaise r\'eduction additive sur~$\Kq$ et l'extension~$\Kl / K$ est non ramifi\'ee en~$\q$.
  Ceci implique qu'en toute place de~$\Kl$ au-dessus de~$\q$,~$E$ a mauvaise r\'eduction additive (\cite{[Sil]} \S VII.5 proposition~5.4).
  On obtient alors une contradiction avec le lemme~\ref{lem:nonredadd}.
\end{proof}

\subsubsection{Lorsque~$j(E)$ est entier en~$\q$}

On suppose dans cette partie que l'invariant~$j$ de~$E$ est entier en~$\q$, c'est-\`a-dire que~$E$ a potentiellement bonne r\'eduction en~$\q$.

La proposition suivante r\'esulte de la discussion de la partie~\ref{sssec:eLbred}.
\begin{proposition}\label{prop:ramhorspbon}
On suppose que~$E$ a potentiellement bonne r\'eduction en~$\q$.
 Alors le groupe~$\la(\Iq)$ est d'ordre~$\eq$~; en particulier, le caract\`ere~$\la^{12}$ est non ramifi\'e en~$\q$.
\end{proposition}

La proposition suivante est le th\'eor\`eme~3~(\S 2) de~\cite{[SeTa]}. 
\begin{proposition}\label{prop:frobhorspbonST}
On suppose que~$E$ a potentiellement bonne r\'eduction en~$\q$. Alors le polyn\^ome caract\'eristique de l'action de~$\sq$ sur le module de Tate en~$p$ de~$E$ est \`a coefficients dans~$\Z$ (et ind\'ependant de~$p$)~; ses racines ont pour valeur absolue complexe~$\sqrt{ N\q}$.
\end{proposition}

\begin{notations}\label{not:Lq}\ 
\begin{enumerate}
\item On note~$\Pq(X)$ le polyn\^ome caract\'eristique de l'action de~$\sq$ sur le module de Tate en~$p$ de~$E$.
Comme le d\'eterminant de la repr\'esentation de~$G_K$ sur le module de Tate en~$p$ de~$E$ est le caract\`ere cyclotomique (\`a valeurs dans~$\Zp^{\times}$),~$\Pq(X)$ est de la forme~$X^2 - \Tq X + N\q$ avec $\Tq$ un entier de valeur absolue inf\'erieure ou \'egale \`a~$2\sqrt{N\q}$. Son discriminant~$\Tq^2 - 4N\q$ est donc un entier n\'egatif et ses racines sont conjugu\'ees complexes l'une de l'autre.
\item On note~$\Lq$ le sous-corps engendr\'e dans~$\C$ par les racines de~$\Pq(X)$~; le corps~$\Lq$ est soit~$\Q$ soit un corps quadratique imaginaire.
\end{enumerate}
\end{notations}

\begin{proposition}\label{prop:frobhorspbon}
On suppose que~$E$ a potentiellement bonne r\'eduction en~$\q$. 
Soit~$\Pcq$ un id\'eal premier de~$\Lq$ au-dessus de~$p$.
Alors les images dans~$\OLq / \Pcq$ des racines de~$\Pq(X)$ sont dans~$\Fpx$~;
 il existe une racine~$\bq$ de~$\Pq(X)$ v\'erifiant~:
 $$
 \left( \la(\sq) , \left( \kip\la^{-1} \right) (\sq) \right) = \left( \bq \Mod \Pcq , \bqb \Mod \Pcq \right).
 $$
\end{proposition}

\begin{proof}
Soit~$\widetilde{\Pq}(X)$ la r\'eduction modulo~$p$ de~$\Pq(X)$.
Alors~$\widetilde{\Pq}(X)$ est le polyn\^ome caract\'eristique de~$\phiEp(\sq)$~; 
il est donc scind\'e dans~$\Fp$ et ses racines sont~$\la(\sq)$ et~$\kip\la^{-1}(\sq)$. 
D'autre part,~$\widetilde{\Pq}(X)$ est aussi la r\'eduction modulo~$\Pcq $ de~$\Pq(X)$, dont les racines dans le corps~$\OLq / \Pcq$ sont les classes modulo~$\Pcq$ des racines de~$\Pq(X)$ dans~$\Lq$.
\end{proof}

\begin{remarques}\label{rem:pdsLq}\ 
\begin{enumerate}
\item Le corps~$\Lq$ est \'egal \`a~$\Q$ si et seulement si le discriminant~$\Tq^2 - 4N\q$ est nul~;
lorsque c'est le cas,~$\Pq(X)$ a une racine double, appartenant \`a~$\Z$, \'egale \`a~$\Tq/2$ et~$N\q$ est le carr\'e de~$\Tq/2$~; 
ceci implique notamment que le degr\'e r\'esiduel de~$\q$ dans l'extension~$K/ \Q$ est pair et que~$E$ a potentiellement bonne r\'eduction supersinguli\`ere en~$\q$.
\item Plus g\'en\'eralement, la courbe~$E$ a potentiellement bonne r\'eduction supersinguli\`ere en~$\q$ si et seulement si $q$ divise $\Tq$. 
Lorsque~$\q$ est de degr\'e~$1$ dans~$K / \Q$,~$q$ divise~$\Tq$ si et seulement si (($\Tq$ = 0) ou ($q = 2$ et ($\Tq = 0$,~$2$ ou~$-2$)) ou ($q = 3$ et ($\Tq = 0$,~$3$ ou~$-3$)).
\item Le polyn\^ome  caract\'eristique de~$\phiEp(\sq)$ est scind\'e dans~$\Fp$ et \'egal \`a $\Pq(X)$ modulo~$p$~; l'entier~$\Tq^2 - 4N\q$ est donc un carr\'e modulo~$p$. On en d\'eduit que soit~$p$ divise~$\Tq^2 - 4N\q$, soit~$\Lq$ est un corps quadratique imaginaire dans lequel~$p$ est d\'ecompos\'e (les deux cas ne s'excluant pas).
\end{enumerate}
\end{remarques}

\section{Compatibilit\'e en et hors de~$p$}\label{sec:compat}

\begin{notations}
Dans toute la suite du texte, on note~$\mu$ la puissance douzi\`eme du caract\`ere d'isog\'enie~$\la$~; on rappelle (notations~\ref{not:dicr}) que~$\Kl$ d\'esigne l'extension ab\'elienne de~$K$ trivialisant le caract\`ere~$\la$ et on note~$\Km$ l'extension de~$K$ trivialisant~$\mu$. Le corps~$\Km$ est inclus dans le corps~$\Kl$ et l'extension~$\Kl / \Km$ est cyclique, de degr\'e divisant le pgcd de~$12$ et~$p-1$~; l'extension~$\Km / K$ est cyclique et d'ordre divisant~$p-1$. On note~$\amu$ le morphisme de groupes injectif du groupe de Galois~$\Gal(\Km/K)$ dans~$\Fpx$ induit par~$\mu$.
\end{notations}

\subsection{Th\'eorie du corps de classes pour le caract\`ere~$\mu$}\label{ssec:cclmu}

La th\'eorie du corps de classes globale appliqu\'ee \`a l'extension ab\'elienne~$\Km/K$ fournit un morphisme de groupes du groupe des id\`eles de~$K$  dans le groupe de Galois~$\Gal(\Km/K)$ qui est continu, surjectif et trivial sur les id\`eles diagonales.
On note~$r$ ce morphisme.

\begin{notations}
On note~$\Ax_{K}$ le groupe des id\`eles de~$K$.
Soit~$\nu$ une place (finie ou infinie) de $K$. On note~$K_{\nu}$ le compl\'et\'e de~$K$ en~$\nu$ et~$r_{\nu}$ la compos\'ee de l'injection de~$K_{\nu}^{\times}$ dans les id\`eles~$\Ax_{K}$ et de l'application de r\'eciprocit\'e~$r$ introduite ci-dessus.
Lorsque~$\nu$ est une place finie, on note~$U_{K_{\nu}}$ les unit\'es du corps local~$K_{\nu}$~; dans ce cas, on utilise indiff\'eremment en indice la place~$\nu$ et l'id\'eal maximal de~$K$ qui lui correspond.
\end{notations}

\subsubsection{En une place infinie}

Soit~$\nu$ une place infinie de~$K$~; alors l'application $r_{\nu}$ est triviale.

En effet, soit~$r'_{\nu}$ le morphisme de groupes de~$K^{\times}_{\nu}$ dans~$\Gal ( \Kl / K )$ associ\'e de mani\`ere analogue \`a l'extension ab\'elienne $\Kl / K$.
Alors $r_{\nu}$ est \'egale \`a la puissance douzi\`eme de la compos\'ee de~$r'_{\nu}$ et de la surjection naturelle de~$\Gal ( \Kl / K )$ dans~$\Gal ( \Km / K )$.
Comme~$\nu$ est une place infinie, l'application~$r'_{\nu}$ a pour image un groupe d'ordre divisant~$2$~; on en d\'eduit que~$r_{\nu}$ est triviale.
 
\subsubsection{En une place hors de~$p$}
 
 Soit~$\q$ un id\'eal maximal de~$K$ qui n'est pas au-dessus de $p$.
  
 D'apr\`es l'\'etude locale men\'ee dans la partie~\ref{sec:etloc}, l'extension~$\Km/K$ est non ramifi\'ee en~$\q$ (propositions~\ref{prop:ramfrobhorspmult} et~\ref{prop:ramhorspbon}).
 Ceci implique que l'application~$r_{\q}$ est triviale sur les unit\'es~$U_{\Kq}$.
 Soit~$\sqb$ la restriction de~$\sq$ \`a~$\Km$~;
 alors~$\sqb$ est l'unique \'el\'ement de Frobenius de~$\Gal(\Km/K)$ associ\'e \`a~$\q$~;
 l'application~$r_{\q}$ envoie toute uniformisante de~$\Kq$ sur~$\sqb$.

\subsubsection{En une place au-dessus de~$p$}
 
 Soit~$\p$ un id\'eal premier de~$K$ au-dessus de~$p$~; alors le morphisme~$\amu \circ r_{\p}$ co\"incide sur les unit\'es~$U_{\Kp}$ avec la compos\'ee suivante~:
$$
U_{\Kp} \xrightarrow{N_{\Kp/\Qp}} U_{\Qp} \xrightarrow[\text{modulo } p]{\text{r\'eduction}} \Fpx \xrightarrow[\text{\`a la puissance } -\ap]{\text{\'el\'evation}} \Fpx.
$$

\subsection{Loi de r\'eciprocit\'e pour le caract\`ere~$\mu$}

\begin{notations} Pour toute place~$\nu$ de~$K$, on note~$\iota_{\nu}$ le plongement de~$K$ dans le compl\'et\'e~$K_{\nu}$. Pour tout id\'eal maximal~$\mathfrak{L}$ de~$K$, on note~$\mathrm{val}_{\mathfrak{L}}$ la valuation de~$K$ associ\'ee \`a~$\mathfrak{L}$ dont l'image est~$\Z$.
\end{notations}

\begin{proposition}\label{prop:recmu}
Soient~$\al$ un \'el\'ement de~$K$ non nul et premier \`a~$p$
 et~$ \prod_{\q \nmid p} \q^{\mathrm{val}_{\q}(\al)}$ la d\'ecomposition de l'id\'eal fractionnaire~$\al \OK$ en produit d'id\'eaux premiers de $K$.
 Alors on a~:
 $$
 \prod\limits_{\q \nmid p} \mu \left(\sq\right)^{\mathrm{val}_{\q}(\al)} =
 \prod_{\p | p} N_{\Kp/\Qp}\left(\iota_{\p}(\al)\right)^{\ap} \Mod p.
 $$
\end{proposition}

\begin{proof}
L'image par $r$ de l'id\`ele principale~$(\iota_{\nu}(\al))_{\nu}$ est triviale.
On d\'etermine l'image de~$\iota_{\nu}(\al)$ par~$\amu \circ r_{\nu}$ pour les diff\'erentes places~$\nu$ de~$K$ en utilisant la partie~\ref{ssec:cclmu}.
\begin{itemize}
 \item[$\bullet$] Si~$\nu$ est une place infinie de~$K$, l'application~$r_{\nu}$ est triviale, donc~$\amu \circ r_{\nu}(\iota_{\nu}(\al))$ l'est \'egalement.
 \item[$\bullet$] Si~$\q$ est un id\'eal maximal de~$K$ premier \`a~$p$, alors~$\iota_{\q}(\al)$ est un \'el\'ement de~$\Kq^{\times}$ de valuation~$\mathrm{val}_{\q}(\al)$~;
  son image par~$r_{\q}$ est~$\sqb^{\mathrm{val}_{\q}(\al)}$ et son image par~$\amu \circ r_{\q}$ est~$ \amu \left(\sqb\right)^{\mathrm{val}_{\q}(\al)}$.
 \item[$\bullet$] Si~$\p$ est un id\'eal maximal de~$K$ au-dessus de~$p$, alors,~$\al$ \'etant suppos\'e premier \`a~$p$,~$\iota_{\p}(\al)$ est une unit\'e de~$\Kp$ ;
  on a  donc~$\amu \circ r_{\p}(\iota_{\p}(\al))  = N_{\Kp/\Qp}(\iota_{\p}(\al))^{-\ap} \Mod  p$.
\end{itemize}
Finalement on a (tous les produits \'etant finis) :
$$
\begin{array}{r c l}
 1 & = & \amu \circ r \left(\left(\iota_{\nu}(\al)\right)_{\nu}\right) \\
   & = & \amu\left(\prod\limits_{\nu}  r_{\nu}\left(\iota_{\nu}(\al)\right)\right) \\
   & = & \prod\limits_{\nu} \amu \circ r_{\nu}\left(\iota_{\nu}(\al)\right) \\
   & =& \prod\limits_{\nu | \infty} 1 \times \prod\limits_{\q \nmid p}\amu \circ r_{\q}\left(\iota_{\q}(\al)\right) \times \prod\limits_{\p | p} \amu \circ r_{\p}\left(\iota_{\p}(\al)\right) \\
   & = & \prod\limits_{\q \nmid p} \amu\left(\sqb\right)^{\mathrm{val}_{\q}(\al)} \times \prod\limits_{\p | p}  \left( N_{K_{\p}/\Q_{\p}}\left(\iota_{\p}(\alpha)\right)^{-\ap} \Mod  p \right)\\
   & = & \prod\limits_{\q \nmid p} \amu\left(\sqb\right)^{\mathrm{val}_{\q}(\al)} \times \left( \left( \prod\limits_{\p | p} N_{K_{\p}/\Q_{\p}}\left(\iota_{\p}(\alpha)\right)^{\ap} \right) \Mod  p \right)^{-1}.\\
\end{array}
$$
Avec l'\'egalit\'e~$\mu\left(\sq\right) = \amu\left(\sqb\right)$ pour tout id\'eal maximal de~$K$ premier \`a~$p$, on obtient le r\'esultat de la proposition.
\end{proof}

\begin{notations}\label{not:po}\ 
\begin{enumerate}
\item Dans tout la suite du texte, on suppose que l'extension~$K/\Q$ est galoisienne ; on note~$G$ son groupe de Galois.
\item On fixe \'egalement pour toute la suite du texte un id\'eal~$\po$ de~$K$ au-dessus de~$p$.
\item Pour tout \'el\'ement~$\tau$ de~$G$, on note~$\ato$ l'entier~$\ap$ associ\'e \`a l'id\'eal~$\p = \tau^{-1}(\po)$.
\item On note~$\cN$ l'application de~$K$ dans lui-m\^eme qui envoie un \'el\'ement~$\alpha$ sur la norme tordue par les entiers~$(\ato)_{\tau \in G}$~:~$\cN (\al) = \prod_{\tau \in G} \tau(\alpha)^{\ato}$. 
On note que l'application~$\cN$ pr\'eserve~$K^{\times}$,~$\OK$ et les \'el\'ements de~$K$ premiers \`a~$p$. 
\end{enumerate}
\end{notations}

Avec ces notations, la proposition~\ref{prop:recmu} admet la reformulation globale suivante (qui redonne le lemme~1 du~\S2 de~\cite{[Mom]}).

\begin{proposition}\label{prop:recmuglob}
 Soient~$\al$ un \'el\'ement de~$K$ non nul et premier \`a~$p$
 et~$\prod_{\q \nmid p} \q^{\mathrm{val}_{\q}(\al)}$ la d\'ecomposition de l'id\'eal fractionnaire~$\al \OK$ en produit d'id\'eaux premiers de $K$.
 Alors on a~:
 $$
 \prod\limits_{\q \nmid p} \mu \left(\sq\right)^{\mathrm{val}_{\q}(\al)} = \iota_{\po} \left( \cN (\al) \right) \Mod \po.
 $$
\end{proposition}

\begin{proof}
On v\'erifie que l'\'el\'ement~$\prod_{\p | p} N_{\Kp/\Qp}\left(\iota_{\p}(\al)\right)^{\ap}$  de~$\Zp$ est l'image par l'injection canonique~$\iota_{\po}$ de~$K$ dans son compl\'et\'e~$K_{\po}$ de l'\'el\'ement~$\prod_{\tau \in G} \tau(\al)^{\ato}$ de~$K$. 
\end{proof}

\subsection{Une borne pour la hauteur des associ\'es d'un entier}

\begin{notation}
Soit~$\al$ un \'el\'ement non nul de~$K$. On note~$\h$ la hauteur absolue de~$\al$ d\'efinie par
$$
\h(\al) =  \left( \prod  \limits_{\nu \text{ place de } K}  \max \left(1, \left| \al \right|_{\nu} \right)\right)^{1/d_K},
$$
avec les normalisations suivantes pour les valeurs absolues~$| \cdot |_{\nu}$~:
\begin{itemize}
\item[$\bullet$] si~$\nu$ est une place r\'eelle, correspondant \`a un \'el\'ement~$\tau$ de~$G$, on pose~$| \cdot |_{\nu} = | \tau(\cdot) |_{\C}$~;
\item[$\bullet$] si~$\nu$ est une place complexe, correspondant \`a un \'el\'ement~$\tau$ de~$G$, on pose~$| \cdot |_{\nu} = | \tau(\cdot) |^2_{\C}$~;
\item[$\bullet$] si~$\nu$ est une place finie, correspondant \`a un id\'eal maximal~$\mathfrak{L}$ de $K$, on pose~$| \cdot |_{\nu} = (N\mathfrak{L})^{ -  \mathrm{val}_{\mathfrak{L}}(\cdot)}$.
\end{itemize}
On remarque que lorsque~$\al$ est entier dans~$K$, les seules places apportant une contribution non triviale dans le produit d\'efinissant~$ \h ( \al ) $ sont les places infinies.
\end{notation}

\begin{lemme}\label{lem:hautalpha}
Soit~$\al$ un \'el\'ement non nul de~$\OK$~; alors, pour tout~$\tau$ dans~$G$, on a~:
$$
\left| \tau \left( \cN(\al) \right) \right|_{\C} \leq \h(\al) ^{12d_K}.
$$
\end{lemme}

\begin{proof}
Soit~$\tau$ dans~$G$ fix\'e~; on a~:
$$
\tau\left( \cN(\al) \right)  =   \tau\left(\prod\limits_{\tau' \in G} \tau'(\al)^{a_{\tau'}} \right) 
				  =  \prod\limits_{\tau' \in G}  \tau'(\al)^{a_{\tau^{-1}\tau'}}, 
$$
Pour tout~$\tau'$ dans~$G$, on a ($a_{\tau^{-1}\tau'}$ \'etant un entier compris entre~$0$ et~$12$))~:
$$
\left|\tau'\left( \al \right) \right|_{\C}^{a_{\tau^{-1}\tau'}}  \leq  \left(\max \left(1,\left|\tau'\left( \al \right) \right|_{\C}\right)\right)^{a_{\tau^{-1}\tau'}} 
                                  \leq  \left(\max \left(1,\left|\tau'\left( \al \right) \right|_{\C}\right)\right)^{12}.
$$
 
Comme le corps~$K$ est suppos\'e galoisien sur~$\Q$, ses places infinies  sont soit toutes  r\'eelles, soit toutes complexes.

Lorsque toutes les places de~$K$ sont r\'eelles, il y en a exactement~$d_K$, qui correspondent bijectivement aux \'el\'ements de~$G$~; on a alors~:
$$
\begin{array}{r c l}
\left|\tau\left( \cN(\al) \right)\right|_{\C}  =  \prod\limits_{\tau' \in G}  \left|\tau'(\al)\right|_{\C}^{a_{\tau^{-1}\tau'}}  & \leq &  \left(\prod\limits_{\tau' \in G} \max \left(1,\left|\tau'\left( \al \right) \right|_{\C}\right)\right)^{12} \\
				                          & \leq &  \left(\prod\limits_{\nu | \infty} \max \left(1,\left| \al \right|_{\nu}\right)\right)^{12} = \h(\al)^{12d_K}.\\
\end{array}
$$

Lorsque toutes les places de~$K$ sont complexes, le degr\'e~$d_K$ de~$K$ sur~$\Q$ est pair et la conjugaison complexe induit dans~$G$ un \'el\'ement~$c$ d'ordre~$2$.
 Le corps~$K$ a exactement~$d_K/2$ places infinies~; deux \'el\'ements de~$G$ d\'efinissent la m\^eme place infinie si et seulement s'ils sont \'egaux ou conjugu\'es complexes l'un de l'autre. 
 On fixe un syst\`eme~$\widetilde{G}$ de repr\'esentants de~$G$ modulo le sous-groupe d'ordre~$2$ engendr\'e par~$c$.
  On a alors~:
$$
\begin{array}{r c l}
\left|\tau\left( \cN(\al) \right)\right|_{\C} & = & \prod\limits_{\tau' \in \widetilde{G}}  \left|\tau'(\al)\right|_{\C}^{a_{\tau^{-1}\tau'}} \left|c\tau'(\al)\right|_{\C}^{a_{\tau^{-1}c\tau'}}\\ 
                                				    & \leq&  \left(\prod\limits_{\tau' \in \widetilde{G}} \max \left(1,\left|\tau'\left( \al \right) \right|_{\C}\right)\max \left(1,\left|c\tau'\left( \al \right) \right|_{\C}\right)\right)^{12} \\
				    			& \leq&  \left(\prod\limits_{\tau' \in \widetilde{G}} \max \left(1,\left|\tau'\left( \al \right) \right|_{\C}\right)^2\right)^{12} \\
							& \leq&  \left(\prod\limits_{\tau' \in \widetilde{G}} \max \left(1,\left|\tau'\left( \al \right) \right|^2_{\C}\right)\right)^{12} \\
				                             & \leq&  \left(\prod\limits_{\nu | \infty} \max \left(1,\left| \al \right|_{\nu}\right)\right)^{12} = \h(\al)^{12d_K}. \\
\end{array}
$$
\end{proof}

\begin{notations}\label{not:C1} Suivant~\cite{[BuGy]}, on note~:
\begin{itemize}
\item[$\bullet$] $R_K$ le r\'egulateur de $K$~;
\item[$\bullet$] $r_K$ le rang du groupe des unit\'es de $K$ ($r_K$ vaut $d_K - 1$ si $K$ est totalement r\'eel et~$\frac{d_K}{2}-1$ sinon)~;
\item[$\bullet$] $\delta_K$ un r\'eel strictement positif minorant~$d_K \ln \left( \h(\al) \right)$ pour  tout \'el\'ement non nul~$\al$ de~$K$ qui n'est pas une racine de l'unit\'e~; 
si~$d_K$ vaut~$1$ ou~$2$, on peut prendre~$\delta_K$ \'egal \`a~$\frac{\ln 2}{r_K +1}$ ; si~$d_K$ est sup\'erieur ou \'egal \`a $3$, on peut prendre~$\delta_K$ \'egal \`a~$\frac{1}{53d_K\ln(6d_K)}$ ou  $\frac{1}{1201}\left(\frac{\ln \left( \ln d_K  \right) }{\ln d_K}\right)^3$.
\end{itemize}
\end{notations}

\begin{definition}[Borne~$C_1(K) $]\label{def:C1(K)}
On pose~:
$$
C_1(K) = \frac{r_K^{r_K +1} \delta_K^{-(r_K -1)}}{2}.
$$
\end{definition}
On remarque que~$C_1(K)$ peut s'exprimer en n'utilisant que le degr\'e~$d_K$ de~$K$ sur~$\Q$. Avec ces notations, le lemme~$2$ (partie~$3$) de~\cite{[BuGy]} s'\'ecrit de la mani\`ere suivante.
 
\begin{lemme}\label{lem:Yann}
Pour tout \'el\'ement non nul~$\al$ de~$\OK$, il existe une unit\'e~$u$ de~$K$ v\'erifiant~:
$$
\h(u\al) \leq \left| N_{K/\Q}(\al) \right|^{1/d_K} \exp\left( C_1(K)R_K \right).
$$
\end{lemme}

\subsection{Type de r\'eduction en une place hors de $p$ et action de l'inertie en~$p$}

Dans cette partie, on suit le raisonnement de la partie~2 de~\cite{[Mom]} pour \'etablir un lien entre les actions sur le groupe d'isog\'enie du frobenius en une place hors de~$p$ d'une part et
des sous-groupes d'inertie aux places au-dessus de~$p$ d'autre part~; 
ce lien est valide lorsque~$p$ est pris strictement plus grand qu'une borne dont on d\'etermine une forme explicite.

\begin{definition}[Borne~$C_2(K)$]\label{def:C2(K)}
On d\'efinit~:
$$
C_2(K) = \exp \left( 12 d_K C_1(K) R_K \right).
$$
\end{definition}
On remarque que le nombre r\'eel~$C_2(K)$ ne d\'epend que du degr\'e et du r\'egulateur de~$K$. On rappelle (notations de l'introduction) que~$h_K$ d\'esigne le nombre de classes d'id\'eaux de~$K$.

\begin{proposition}\label{prop:hautgen}
 Soit~$\mathfrak{L}$ un id\'eal maximal de~$K$.
Il existe un g\'en\'erateur~$\gamma_{\mathfrak{L}}$ de~$\mathfrak{L}^{h_K}$ satisfaisant pour tout~$\tau$ dans~$G$~:
$$
\left| \tau \left( \cN(\gamma_{\mathfrak{L}}) \right) \right|_{\C} \leq \left(N \mathfrak{L} \right)^{12h_K} C_2(K).
$$
\end{proposition}

\begin{proof}
\'Etant donn\'e un g\'en\'erateur quelconque de~$\mathfrak{L}^{h_K}$, on peut le multiplier par une unit\'e donn\'ee par le lemme~\ref{lem:Yann} pour obtenir un g\'en\'erateur~$\gamma_{\mathfrak{L}}$ qui v\'erifie~:
$$
\h(\gamma_{\mathfrak{L}}) \leq \left| N_{K/\Q}(\gamma_{\mathfrak{L}}) \right|^{1/d_K} \exp\left( C_1(K)R_K \right).
$$
Comme~$\gamma_{\mathfrak{L}}$ engendre dans~$\OK$ l'id\'eal~$\mathfrak{L}^{h_K}$, la norme de~$\gamma_{\mathfrak{L}}$ dans l'extension~$K/\Q$ est un entier relatif \'egal ou oppos\'e \`a~$(N\mathfrak{L})^{h_K}$.
D'apr\`es le lemme \ref{lem:hautalpha}, on a pour tout~$\tau$ dans~$G$~:
$$
\left|\tau\left( \cN(\gamma_{\mathfrak{L}}) \right)\right|_{\C} \leq \h(\gamma_{\mathfrak{L}})^{12d_K},
$$
d'o\`u
$$
\begin{array}{r c l}
\left|\tau\left( \cN(\gamma_{\mathfrak{L}}) \right)\right|_{\C} & \leq & \left| N_{K/\Q}(\gamma_{\mathfrak{L}}) \right|^{12} \left(\exp\left(C_1(K)R_K \right)\right)^{12d_K} \\
                                                              & \leq & \left(N\mathfrak{L} \right)^{12h_K} \exp \left(12d_KC_1(K)R_K \right) =  \left(N\mathfrak{L} \right)^{12h_K} C_2(K). \\
\end{array}
$$
\end{proof}

\begin{notations}\label{not:poq}
 Soit~$\q$ un id\'eal maximal de~$K$ premier \`a~$p$ en lequel~$E$ a potentiellement bonne r\'eduction.
\begin{enumerate}
\item On fixe un id\'eal~$\poq$ de~$K\Lq$ au-dessus de~$\po$ (voir les notations~\ref{not:Lq} et~\ref{not:po}).
Comme le corps~$\Lq$ est soit~$\Q$, soit un corps quadratique imaginaire, il y a au plus deux choix possibles pour~$\poq$~;
 si~$\Lq$ est inclus dans~$K$, le seul choix possible pour~$\poq$ est~$\poq = \po$.
\item On note~$\Poq$ l'unique id\'eal de~$\Lq$ situ\'e au-dessous de~$\poq$.
\item D'apr\`es la proposition~\ref{prop:frobhorspbon}, il existe une racine du polyn\^ome~$\Pq(X)$ dont la classe modulo~$\Poq$ (qui est dans~$\Fp$) vaut~$\la(\sq)$~; on note~$\bq$ une telle racine.
\end{enumerate}
\end{notations}
 
\begin{definition}[Borne~$C(K , n)$]\label{def:C(K,n)}
Pour tout entier~$n$, on pose~:
$$
C(K , n) =  \left( n^{12h_K} C_2(K) + n^{6h_K} \right)^{2d_K}.
$$
\end{definition}

\begin{proposition}\label{prop:egaglob}
Soient~$\q$ un id\'eal maximal de~$K$ premier \`a~$p$ et~$\gq$ un g\'en\'erateur de~$\q^{h_K}$ v\'erifiant l'in\'egalit\'e de la proposition~\ref{prop:hautgen}.
On suppose que~$p$ est strictement sup\'erieur \`a~$C ( K, N\q )$.
Alors on est dans l'un des trois cas suivants (avec les notations~\ref{not:po} et~\ref{not:poq})~:
 \begin{center}
 \renewcommand\arraystretch{1.5}
\begin{tabular}{| c | c | c | c |} \hline 
Type de r\'eduction semi-stable de~$E$ en~$\q$ & $\mu(\sq)$ & $\cN(\gq)$  & Cas\\ \hline 
\multirow{2}{*}{multiplicatif}  & $1 \Mod p$ & $\cN(\gq) = 1$ & M0 \\  \cline{2-4} 
    & $(N\q)^{12} \Mod p$  & $\cN(\gq) = (N\q)^{12h_K}$ &  M1\\  \cline{2-4} \hline
  bon  & $\bq^{12} \Mod \Poq$ & $\cN(\gq) = \bq^{12h_K}$ & B \\ \hline 
  \end{tabular}
  \end{center} 
\end{proposition}

\begin{proof}
Comme~$\gq$ engendre dans~$\OK$ l'id\'eal~$\q^{h_K}$, la proposition~\ref{prop:recmuglob} appliqu\'ee \`a~$\gq$ donne~:
$$
\mu(\sq )^{h_K} = \iota_{\po} \left( \cN (\gq) \right) \Mod \po
$$
On raisonne ensuite selon le type de r\'eduction semi-stable de~$E$ en~$\q$ et les valeurs possibles de~$\mu(\sq)$ d\'etermin\'ees dans la partie~\ref{ssec:ramfrobhorsp}.

On suppose d'abord que~$E$ a potentiellement mauvaise r\'eduction multiplicative en~$\q$ et que~$\la^2 (\sq)$ est \'egal \`a~$1$ modulo~$p$.
 Alors~$\mu (\sq)$ est aussi \'egal \`a~$1$ modulo~$p$ et l'\'el\'ement~$\cN(\gq)$ de~$\OK$ est congru \`a $1$ modulo~$\po$. 
  Comme~$\po$ est au-dessus de~$p$, ceci implique que~$p$ divise l'entier relatif~$N_{K/\Q}\left(\cN(\gq) - 1\right)$. 
  Par choix de~$\gq$, on a pour tout~$\tau$ dans~$G$~:
$$
\left| \tau\left( \cN(\gq) - 1 \right) \right|_{\C}  \leq  \left|\tau\left( \cN(\gq) \right) \right|_{\C} + \left|\tau\left( 1 \right) \right|_{\C} 
 							            \leq  (N\q)^{12h_K}C_2(K) + 1,
$$
d'o\`u
$$
\left|N_{K/\Q}\left(\cN(\gq) - 1 \right)\right|_{\Q}    =   \prod\limits_{\tau \in G}  \left| \tau\left( \cN(\gq) - 1 \right) \right|_{\C} 
									 \leq  \left( (N\q)^{12h_K}C_2(K) + 1 \right)^{d_K}
									 \leq  C(K,N\q).
$$
Comme $p$ est suppos\'e strictement plus grand que $C(K,N\q)$, on en d\'eduit que l'entier~$N_{K/\Q}\left(\cN(\gq) - 1 \right)$ est nul, ce qui implique que~$\cN(\gq)$ est \'egal \`a~$1$.

On suppose ensuite que~$E$ a potentiellement mauvaise r\'eduction multiplicative en~$\q$ et que~$\la^2 (\sq)$ est \'egal \`a~$(N\q)^2$ modulo~$p$.
Alors~$\mu (\sq)$ est  \'egal \`a~$(N\q)^{12}$ modulo~$p$ et l'\'el\'ement~$\cN(\gq)$ de~$\OK$ est congru \`a $(N\q)^{12h_K}$ modulo~$\po$. 
  Ainsi,~$p$ divise l'entier relatif~$N_{K/\Q}\left(\cN(\gq) - (N\q)^{12h_K}\right)$. 
Or on a :
$$
\begin{array}{r c l}
\left|N_{K/\Q}\left(\cN(\gq) - (N\q)^{12h_K} \right)\right|_{\Q} & = & \prod\limits_{\tau \in G} \left| \tau\left( \cN(\gq) - (N\q)^{12h_K} \right) \right|_{\C} \\
									& \leq & \prod\limits_{\tau \in G} \left( (N\q)^{12h_K}C_2(K) + (N\q)^{12h_K} \right) \\
									& \leq & \left( (N\q)^{12h_K}C_2(K) + (N\q)^{12h_K} \right)^{d_K} \\
									& \leq & C(K,N\q).\\
\end{array}
$$
Comme $p$ est suppos\'e strictement plus grand que $C(K,N\q)$, on en d\'eduit que l'entier~$N_{K/\Q}\left(\cN(\gq) - (N\q)^{12h_K} \right)$ est nul, ce qui implique que~$\cN(\gq)$ est \'egal \`a~$(N\q)^{12h_K}$.

On suppose enfin que~$E$ a potentiellement bonne r\'eduction en~$\q$.
Alors~$\mu(\sq)$ est \'egal \`a la classe de~$\bq^{12}$ modulo~$\Poq$ (notations~\ref{not:poq}) et~$\cN(\gq)$ et~$\bq^{12h_K}$ sont congrus modulo l'id\'eal~$\poq$ de~$K\Lq$.
Ceci implique que~$p$ divise l'entier relatif~$N_{K\Lq/\Q}\left(\cN(\gq) - \bq^{12h}\right)$.

Comme~$\Lq$ est soit~$\Q$ soit un corps quadratique imaginaire, le corps~$K\Lq$ est galoisien sur~$\Q$, de degr\'e \'egal \`a~$d_K$ ou~$2d_K$.
 Tout  \'el\'ement~$\tau$ de~$\Gal(K\Lq/\Q)$ induit sur~$K$ un \'el\'ement de~$\Gal(K/\Q)$  et sur~$\Lq$ un \'el\'ement de~$\Gal(\Lq/\Q)$, qui est donc soit l'identit\'e soit la conjugaison complexe. 
On a ainsi pour tout~$\tau$ dans~$\Gal(K\Lq/\Q)$~:
$$\begin{array}{r c l}
\left|\tau \left( \cN(\gq) - \bq^{12h_K} \right) \right|_{\C} & \leq & \left|\tau_{|K}\left(\cN(\gq)\right) \right|_{\C} + \left|\tau_{|\Lq}\left(\bq^{12h_K}\right) \right|_{\C}\\
                                    & \leq & (N\q)^{12h_K}C_2(K) + \left(\sqrt{N\q}\right)^{12h_K}\\
				    & \leq &  (N\q)^{12h_K}C_2(K) + (N\q)^{6h_K}.\\
\end{array}$$
D'o\`u finalement~:
$$
\begin{array}{r c l}
\left|N_{K\Lq/\Q}\left(\cN(\gq) - \bq^{12h_K}\right)\right|_{\Q} & = &\prod\limits_{\tau \in \Gal(K\Lq/\Q)} \left|\tau\left(\cN(\gq) - \bq^{12h_K}\right) \right|_{\C} \\
                                                                    & \leq & \left((N\q)^{12h_K}C_2(K) + (N\q)^{6h_K}\right)^{2d_K} = C(K , N\q).\\
								    
\end{array}$$
Comme~$p$ est suppos\'e strictement plus grand que~$C(K , N \q)$, on en d\'eduit que l'entier~$N_{K\Lq/\Q}\left(\cN(\gq) - \bq^{12h_K} \right)$ est nul, ce qui implique que~$\cN(\gq)$ est \'egal \`a~$\bq^{12h_K}$.
\end{proof}

La proposition suivante donne une forme effective du lemme~2 de la partie~2 de~\cite{[Mom]}.

\begin{proposition}\label{prop:compat}
Soient~$q$ un nombre premier rationnel totalement d\'ecompos\'e dans~$K$ et~$\q$ un id\'eal premier de~$K$ au-dessus de~$q$.
On suppose que~$p$ est strictement plus grand que $C(K, q)$.
Alors on est dans l'un des cinq cas suivants (avec les notations~\ref{not:poq})~:
 \begin{center}
 \renewcommand\arraystretch{1.5}
 \noindent
 \begin{small} 
 \begin{tabular}{|cc|c|c|c|c|} \hline
  \multicolumn{2}{|c|}{Type de r\'eduction} & \multirow{2}{*}{Corps $\Lq$} & \multirow{2}{*}{$\mu(\sq)$} & \multirow{2}{*}{Famille $(\ato)_{\tau \in G}$}& \multirow{2}{*}{Cas} \\ 
  \multicolumn{2}{|c|}{semi-stable de~$E$ en~$\q$} & & & & \\ \hline 
  \multicolumn{2}{|c|}{\multirow{2}{*}{multiplicatif}} & \multirow{2}{*}{} & $1 \mod p$  & $\forall \tau \in G, \ato = 0$ &  M0  \\ \cline{4-6} 
  & & & $q^{12} \mod p$ & $\forall \tau \in G, \ato = 12$  & M1 \\  \hline
  \multicolumn{1}{|c|}{\multirow{5}{*}{bon}} & supersingulier & $\Lq = \Q(\sqrt{-q})$ & $\bq^6 = - q^3$ & $\forall \tau \in G, \ato = 6$ & BS \\ \cline{2-6}
   & \multicolumn{1}{|c|}{\multirow{4}{*}{ordinaire}} & \multirow{1}{*}{quadratique,} & \multirow{2}{*}{$ N_{K/\Lq}(\q) = \bq\OLq$} & $\forall \tau \in \Gal(K/\Lq), \ato = 12$ & \multirow{2}{*}{BO} \\  
   & \multicolumn{1}{|c|}{} & inclus dans~$K$, & & et $\ato = 0$ sinon & \\ \cline{4-6} 
  & \multicolumn{1}{|c|}{}& \multirow{1}{*}{$p$ d\'ecompos\'e} & \multirow{2}{*}{$ N_{K/\Lq}(\q) = \bqb\OLq$} & $\forall \tau \in \Gal(K/\Lq), \ato = 0$ & \multirow{2}{*}{BO'} \\ 
  & \multicolumn{1}{|c|}{} & dans $\Lq$ & & et $\ato = 12$ sinon & \\  \hline 
  \end{tabular} 
  \end{small}
  \end{center}
 \end{proposition}
 
 \begin{proof}
Comme le nombre premier~$q$ est suppos\'e totalement d\'ecompos\'e dans l'extension galoisienne~$K/\Q$, les id\'eaux~$\tau(\q)$ sont deux \`a deux distincts lorsque~$\tau$ d\'ecrit~$G$, leur produit est l'id\'eal~$q\OK$ 
et la norme de~$\q$ dans~$K / \Q$ est \'egale \`a~$q$.
En particulier, en supposant que~$p$ est strictement sup\'erieur \`a~$C(K, q)$, on se place dans les hypoth\`eses de la proposition~\ref{prop:egaglob}.
On raisonne donc selon les trois cas de cette proposition.

 On remarque que l'\'el\'ement~$\cN(\gq)$ engendre dans~$\OK$ l'id\'eal
$$
\begin{array}{r c c c l }
  \cN(\gq) \OK &  = & \left(\prod\limits_{\tau \in G} \tau(\gq)^{\ato}\right) \OK  &  =  & \prod\limits_{\tau \in G} \left(\tau\left(\gq\OK\right)\right)^{\ato} \\ \\
  			&  = & \prod\limits_{\tau \in G} \left( \tau \left( \q^{h_K} \right) \right)^{\ato} &  = & \left( \prod\limits_{\tau \in G} \tau \left(  \q \right)^{\ato} \right)^{h_K}.\\
 \end{array}
  $$

Lorsque~$\q$ est de type M0,~$\cN(\gq)$ est \'egal \`a~$1$ donc engendre dans~$\OK$ l'id\'eal~$\OK$ lui-m\^eme. L'\'egalit\'e
$$
\OK = \left(\prod\limits_{\tau \in G} \tau\left( \q \right)^{\ato} \right)^{h_K}
$$
avec $h_K$ dans~$\mathbb{N}^{*}$ implique alors que pour tout~$\tau$ dans~$G$,~$\ato$ est nul.
 
 Lorsque~$\q$ est de type M1,~$\cN(\gq)$ est \'egal \`a~$q^{12h_K}$ et on a~:
 $$
 \left(\prod\limits_{\tau \in G} \tau\left( \q \right)^{\ato} \right)^{h_K}   =  \cN(\gq) \OK    =   \left(q^{12h_K}\right) \OK  =   \left(q \OK \right) ^{12h_K}  =   \left(\prod\limits_{\tau \in G} \tau\left( \q \right) \right)^{12h_K}.
$$
 On en d\'eduit que pour tout~$\tau$ dans~$G$,~$\ato$ est \'egal \`a~$12$.

Lorsque~$\q$ est de type B,~$\cN(\gq)$ est \'egal \`a $\bq^{12h_K}$ (dans le corps $K\Lq$).
 En particulier, l'\'el\'ement~$\cN(\gq)$ de~$K$ est contenu dans~$\Lq$ et deux sous-cas sont possibles~:
  soit~$\cN(\gq)$ est rationnel, soit~$\cN(\gq)$  engendre~$\Lq$ qui est alors contenu dans~$K$.

On suppose d'abord que~$\cN(\gq)$ est rationnel.
Alors~$\bq^{12h_K}$ est \'egalement rationnel, donc~$\bq^{12h_K}$ est \'egal \`a $\bqb^{12h_K}$.
Ceci implique qu'il existe une racine~$12h_K$-i\`eme de l'unit\'e~$\zeta$ dans~$\Lq$ telle que~$\bqb$ est \'egal \`a~$\zeta\bq$.
Comme~$\Lq$ est soit~$\Q$ soit un corps quadratique imaginaire,~$\zeta$ est en fait une racine deuxi\`eme, quatri\`eme ou sixi\`eme de l'unit\'e.
 Ceci implique que~$\bq^{12}$ est d\'ej\`a rationnel~; comme~$\bq$ est un entier alg\'ebrique,~$\bq^{12}$ est un entier relatif~; sa valeur absolue \'etant~$q^6$, on a~$\bq^{12} = \pm q^6$.
 On en d\'eduit qu'on a dans~$\OK$~:
$$
 \left(\prod\limits_{\tau \in G} \tau\left( \q \right)^{\ato} \right)^{h_K}    =     \cN(\gq) \OK  =  \bq^{12h_K} \OK  =  q^{6h_K}\OK = \left(\prod\limits_{\tau \in G} \tau\left( \q \right) \right)^{6h_K}.
  $$  
 Ceci implique que pour tout~$\tau$ dans~$G$, $\ato$ est \'egal \`a~$6$. D'apr\`es la proposition~\ref{prop:inpbon}, le nombre premier~$p$ est donc congru \`a~$3$ modulo~$4$.
 
 On va maintenant montrer que~$E$ a potentiellement r\'eduction supersinguli\`ere en~$\q$, que le corps~$\Lq$ est~$\Q(\sqrt{-q})$ et qu'on a~$\bq^{6} = - q^3$.

La relation
$$
q = N \q =  \bqb \bq =  \zeta\bq^2.
$$
implique dans $\OLq$
$$
q \OLq = \left(\bq^2\right)\OLq = \left( \bq \OLq \right)^2.
$$
Ainsi,~$q$ est ramifi\'e dans~$\Lq$ et l'unique id\'eal premier de~$\Lq$ au-dessus de~$q$ est~$\bq \OLq$ (ceci force notamment~$\Lq$ \`a \^etre diff\'erent de~$\Q$).

On traite d'abord le cas~$q = 2$.
La trace~$\Tq$ (voir notations~\ref{not:Lq}) est alors un entier relatif de valeur absolue inf\'erieure \`a~$2\sqrt{2}$~; il vaut donc~$0$,~$1$,~$2$,~$-1$ ou~$-2$.
Les valeurs correspondantes du discriminant~$\Tq^2 - 4N\q$ du polyn\^ome~$\Pq(X)$ sont~$-8$,~$-7$,~$-4$,~$-7$,~$-4$.
Les entiers relatifs sans facteurs carr\'es dont une racine carr\'ee engendre~$\Lq$ sont alors respectivement~$-2$,~$-7$,~$-1$,~$-7$,~$-1$.
 Or,~$2$ est ramifi\'e dans~$\Lq$ si et seulement si ce dernier entier est congru \`a~$2$ ou~$3$ modulo~$4$. Les valeurs de~$\Tq$ qui r\'ealisent cette condition sont~$0$,~$2$ et~$-2$.
Or d'apr\`es la remarque~\ref{rem:pdsLq}, l'entier~$\Tq^2 - 4N\q$ est un carr\'e modulo~$p$.
Si~$\Tq$ est \'egal \`a~$2$ ou~$-2$, le discriminant~$\Tq^2 - 4N\q$ est \'egal \`a~$-4$.
Or,~$p$ \'etant congru \`a~$3$ modulo~$4$,~$-4$ n'est pas un carr\'e modulo~$p$.
On en d\'eduit que~$\Tq$ est nul.

Lorsque~$q$ est diff\'erent de~$2$, le fait qu'il soit ramifi\'e dans~$\Lq$ implique qu'il divise le discriminant~$\Tq^2 - 4N\q$.
 Comme~$q$ divise~$N\q$ (et lui est m\^eme ici \'egal), on en d\'eduit que~$q$ divise~$\Tq$.
 Alors~$\Tq^2$ est un carr\'e, multiple de~$q$, compris entre~$0$ et~$4q$. Ceci implique que soit~$\Tq$ est nul, soit~$q$ est \'egal \`a~$3$ et~$\Tq^2$ est \'egal \`a~$9$.

On a ainsi montr\'e~:
\begin{itemize}
 \item[$\bullet$] si~$q$ est diff\'erent de~$3$, alors~$\Tq$ est nul~;
 \item[$\bullet$] si~$q$ est \'egal \`a~$3$, alors~$\Tq^2$ est soit nul, soit \'egal \`a~$9$.
\end{itemize}
Dans les deux cas,~$q$ divise~$\Tq$, donc la courbe elliptique obtenue sur $\overline{\kq}$ par r\'eduction de~$E$ est supersinguli\`ere.

Si la trace~$\Tq$ est nulle, alors le polyn\^ome~$\Pq(X)$ dont~$\bq$ est racine est \'egal \`a~$X^2 + q$.
 On a alors~$\bq^{2}$ \'egal \`a~$-q$, donc~$\bq^6$ est \'egal \`a~$-q^3$ et le corps~$\Lq$ est~$\Q(\sqrt{-q})$.

Si~$q$ est \'egal \`a~$3$ et~$\Tq^2$ est \'egal \`a~$9$, alors~$\Tq^2 - 4N\q$ vaut~$-3$ et il existe~$\varepsilon_{\q}$ valant~$+1$ ou~$-1$ v\'erifiant~:
$$
\bq = \frac{\Tq + i \varepsilon_{\q} \sqrt{3} }{2}.
$$
On a alors encore~$\Lq$ \'egal \`a~$\Q(\sqrt{-3})$ et on v\'erifie par le calcul que~$\bq^6$ est \'egal \`a~$-3^3$.

 On suppose enfin que l'\'el\'ement~$\cN(\gq)$ de~$K$ engendre~$\Lq$ qui est un corps quadratique imaginaire.
Alors~$\Lq$ est inclus dans~$K$, le degr\'e~$d_K$ de~$K$ sur~$\Q$ est pair et le groupe de Galois de~$K$ sur~$\Q$ contient comme sous-groupe d'indice~$2$ le groupe~$\Gal(K/\Lq)$~; on note~$\Hq$ ce sous-groupe.  
 
 Comme~$\cN(\gq)$ est contenu dans~$\Lq$, on a pour tout \'el\'ement~$\rho$ de~$\Hq$~:
 $$
\cN(\gq)   =  \rho\left(\cN(\gq)\right). 
 $$
On en d\'eduit~:
 $$
   \left(\prod\limits_{\tau \in G} \tau(\q)^{\ato}\right)^{h_K} =   \cN(\gq)\OK    =   \rho\left( \cN(\gq) \OK  \right) =   \rho\left( \left(\prod\limits_{\tau \in G} \tau(\q)^{\ato}\right)^{h_K} \right) 
     =  \left(\prod\limits_{\tau \in G} \tau(\q)^{a_{\rho^{-1}\tau}}\right)^{h_K} .
  $$
  On a donc pour tout~$\rho$ dans~$\Hq$ et tout~$\tau$ dans~$G$ l'\'egalit\'e~$a_{\rho\tau} = \ato$.
  Ainsi la valeur de~$\ato$ est constante sur les classes \`a gauche (qui sont aussi les classes \`a droite) de~$G$ modulo~$\Hq$.
 
  Le groupe~$G$ poss\`ede deux classes modulo~$\Hq$~: celle, \'egale \`a~$\Hq$, de l'identit\'e et celle, \'egale \`a~$\Hq\gamma$, d'un \'el\'ement~$\gamma$ de~$G$ qui induit la conjugaison complexe sur~$\Lq$.
  On a alors~: 
  $$
  \begin{array}{r c  l}
  \cN(\gq) &  = & \prod\limits_{\tau \in G} \tau(\gq)^{\ato}\\
		      & = & \left(\prod\limits_{\tau \in \Hq} \tau(\gq)\right)^{a_{id}} \left(\prod\limits_{\tau \in \gamma\Hq} \tau(\gq)\right)^{a_\gamma} \\
		      & = & \left(\prod\limits_{\tau \in \Hq} \tau(\gq)\right)^{a_{id}} \left(\gamma \left(\prod\limits_{\tau \in \Hq} \tau(\gq)\right)\right)^{a_\gamma}\\
		      & = & \left(N_{K/\Lq}(\gq)\right)^{a_{id}} \left(\gamma\left(N_{K/\Lq}(\gq)\right)\right)^{a_\gamma} \\
		      & = & \left(N_{K/\Lq}(\gq)\right)^{a_{id}} \left(\overline{N_{K/\Lq}(\gq)}\right)^{a_\gamma}. \\
  \end{array}
  $$
  Comme~$N_{K/\Lq}(\gq)$ engendre dans~$\OLq$ l'id\'eal
  $$
  N_{K/\Lq}(\gq)\OLq   =  N_{K/\Lq}\left(\gq\OK\right) 
                                                  =  N_{K/\Lq}\left(\q^{h_K}\right) 
                                                  =  N_{K/\Lq}\left(\q\right)^{h_K}, 
  $$
  alors~$\overline{N_{K/\Lq}(\gq)}$ engendre dans~$\OLq$ l'id\'eal~$\overline{N_{K/\Lq}\left(\q\right)}^{h_K}$ et on obtient~:
  $$
  \begin{array}{r c l}
  \left(\bq^{12h_K}\right)\OLq & = & \cN(\gq)\OLq \\
  \left(\bq \OLq \right)^{12h_K}& = &\left(N_{K/\Lq}(\gq)\OLq\right)^{a_{id}} \left(\overline{N_{K/\Lq}(\gq)} \OLq\right)^{a_\gamma} \\
 					     & = &\left(N_{K/\Lq}\left(\q\right)^{h_K}\right)^{a_{id}} \left(\overline{N_{K/\Lq}\left(\q\right)}^{h_K}\right)^{a_\gamma}\\
 				    	     & = & \left(\left(N_{K/\Lq}\left(\q\right)\right)^{a_{id}} \left(\overline{N_{K/\Lq}\left(\q\right)}\right)^{a_\gamma}\right)^{h_K}.\\
  \end{array}
  $$

  Comme on a suppos\'e~$q$ totalement  d\'ecompos\'e dans l'extension~$K/\Q$,~$q$ est \'egalement totalement d\'ecompos\'e dans les extensions~$\Lq/\Q$ et~$K/\Lq$.
   La norme~$N_{K/\Lq}(\q)$ est donc un id\'eal premier de~$\Lq$ au-dessus de~$q$.
   Or l'\'egalit\'e~$q = \bq\bqb$ indique que les deux id\'eaux premiers (distincts) de~$\Lq$ au-dessus de~$q$ sont~$\bq\OLq$ et~$\bqb\OLq$. 
  Ainsi~$N_{K/\Lq}(\q)$  est \'egal soit \`a~$\bq\OLq$ soit \`a~$\bqb\OLq$.
  
  Si~$N_{K/\Lq}(\q)$  est \'egal \`a~$\bq\OLq$,  alors la relation
  $$
  \left(\bq \OLq \right)^{12h_K} = \left(\left( \bq\OLq \right)^{a_{id}} \left( \bqb\OLq \right)^{a_\gamma} \right)^{h_K}
  $$
   implique~$a_{id} = 12$ et~$a_\gamma = 0$. On en d\'eduit que~$\ato$ vaut $12$ si~$\tau$ est dans~$\Hq$ et~$0$ sinon.
  
  Si~$N_{K/\Lq}(\q)$  est \'egal \`a~$\bqb\OLq$,  alors la relation
  $$
  \left(\bq \OLq \right)^{12h_K} = \left(\left( \bqb\OLq \right)^{a_{id}} \left( \bq\OLq \right)^{a_\gamma} \right)^{h_K}
  $$
   implique~$a_{id} = 0$ et~$a_\gamma = 12$. On en d\'eduit que~$\ato$ vaut~$0$ si~$\tau$ est dans~$\Hq$ et~$12$ sinon.
   
   D'apr\`es la remarque~\ref{rem:pdsLq}, soit~$p$ est d\'ecompos\'e dans~$\Lq$, soit~$p$ divise~$\Tq^2 - 4 q$.
    Supposons par l'absurde que~$p$ divise~$\Tq^2 - 4 q$~; on remarque que la valeur absolue de~$\Tq^2 - 4 q$ est~$4 q - \Tq^2 $, elle donc inf\'erieure ou \'egale \`a~$4q$.
    Comme~$p$ est suppos\'e strictement sup\'erieur \`a~$C(K , q)$,~$p$ est strictement sup\'erieur \`a~$4q$. On obtient alors que~$\Tq^2 - 4 q$ est nul, ce qui est impossible.
    On en d\'eduit que~$p$ est d\'ecompos\'e dans le corps quadratique~$\Lq$.
    
  Enfin, on suppose par l'absurde que~$E$ a potentiellement r\'eduction supersinguli\`ere en~$\q$~; 
  alors d'apr\`es la remarque~\ref{rem:pdsLq}, on est dans l'un des cas suivants~:~$\Tq$ est nul~;~$q$ est \'egal \`a~$2$ et~$\Tq$ est \'egal \`a~$0$,~$2$ ou~$-2$~;~$q$ est \'egal \`a~$3$ et~$\Tq$ est \'egal \`a~$0$,~$3$ ou~$-3$. 
  Si~$\Tq$ est nul ou~$q$ est \'egal \`a~$3$ et~$\Tq$ vaut~$\pm 3$, alors~$\Lq$ est \'egal \`a~$\Q(\sqrt{-q})$, dans lequel~$q$ est ramifi\'e~;
  si~$q$ est \'egal \`a~$2$ et~$\Tq$ vaut~$\pm 2$, alors~$\Lq$ est \'egal \`a~$\Q( i )$, dans lequel~$2$ est ramifi\'e.
  Or, on a montr\'e que~$q$ est d\'ecompos\'e dans le corps~$\Lq$. On en d\'eduit que~$E$ a potentiellement r\'eduction ordinaire en~$\q$.
  \end{proof}

\section{Deux crit\`eres d'irr\'eductibilit\'e}\label{sec:critirr}

\subsection{Types de r\'eduction semi-stable diff\'erents}
Soit~$M$ un entier naturel sup\'erieur ou \'egal~$1$.

\begin{definition}[Ensemble~$\mathcal{E}(K ; M)$]
On note~$\mathcal{E}(K ; M)$ l'ensemble des courbes elliptiques~$E$ d\'efinies sur~$K$ pour lesquelles il existe
\begin{itemize}
\item[$\bullet$] $q$ et~$q'$ des nombres premiers totalement d\'ecompos\'es dans~$K$ et inf\'erieurs ou \'egaux \`a~$M$,
\item[$\bullet$] une place~$\q$ de~$K$ au-dessus de~$q$ et une place~$\q'$ de~$K$ au-dessus de~$q'$ tels que les types de r\'eduction semi-stable de~$E$ en~$\q$ et~$\q'$ (de la proposition~\ref{prop:compat}) sont diff\'erents
\end{itemize}
 (les premiers~$q$ et~$q'$ et les places~$\q$ et~$\q'$ peuvent d\'ependre de~$E$~; les premiers~$q$ et~$q'$ peuvent \^etre \'egaux).
\end{definition}

\begin{proposition}\label{prop:critirrqq'}
Soit~$M$ un entier naturel sup\'erieur ou \'egal~$1$.
Alors pour toute courbe elliptique~$E$ dans~$\mathcal{E}(K ; M)$ et tout nombre premier~$p$ strictement sup\'erieur \`a~$C(K , M)$, la repr\'esentation~$\varphi_{E,p}$ est irr\'eductible.
\end{proposition}

\begin{proof}
On suppose par l'absurde que la repr\'esentation~$\phiEp$ est r\'eductible.
Alors pour les places~$\q$ et~$\q'$, on est dans les hypoth\`eses de la proposition~\ref{prop:compat}.
Le type de r\'eduction semi-stable de~$E$ en~$\q$ ou~$\q'$ d\'etermine donc la famille~$(\ato)_{\tau}$, qui d\'ecrit l'action sur le sous-groupe d'isog\'enie des sous-groupes d'inertie de~$G_K$ aux places de~$K$ au-dessus de~$p$.
Comme deux types de r\'eduction semi-stable pour~$\q$ ou~$\q'$ donnent des familles~$(\ato)_{\tau}$ diff\'erentes 
(on distingue les cas BO et BO' par le fait que dans le cas BO, l'ensemble des \'el\'ements~$\tau$ de~$G$ pour lesquels~$\ato$ est \'egal \`a~$12$ est un sous-groupe de~$G$ alors que dans le cas BO', c'est le compl\'ementaire d'un sous-groupe), ces types doivent \^etre les m\^emes~: on obtient une contradiction.
\end{proof}

\subsection{R\'eduction semi-stable multiplicative}

Soit~$q$ un nombre premier rationnel totalement d\'ecompos\'e dans~$K$.

\begin{definition}[Ensemble~$\mathcal{E'}(K ; q)$ et borne~$B(K ; q)$]
 Soit~$q$ un nombre premier rationnel totalement d\'ecompos\'e dans~$K$. 
 \begin{enumerate}
\item On note~$\mathcal{E'}(K ; q)$ l'ensemble des courbes elliptiques~$E$ d\'efinies sur~$K$ v\'erifiant~: il existe une place  (qui peut d\'ependre de~$E$) de~$K$ au-dessus de~$q$ en laquelle~$E$ a potentiellement mauvaise r\'eduction multiplicative. 
\item On pose
$$
B(K ; q) = \max \left( C(K,q) ,  \left(  1 + 3^{6 d_K h_K}  \right)^2 \right).
$$
\end{enumerate}
\end{definition}

\begin{proposition}\label{prop:critirrq}
Soit~$q$ un nombre premier rationnel totalement d\'ecompos\'e dans~$K$.
Alors pour toute courbe elliptique~$E$ dans~$\mathcal{E'}(K ; q)$ et tout nombre premier~$p$ strictement sup\'erieur \`a~$B(K ; q)$, la repr\'esentation~$\varphi_{E,p}$ est irr\'eductible.
\end{proposition}

\begin{proof}
On suppose par l'absurde que la repr\'esentation~$\phiEp$ est r\'eductible.
Alors on est dans les hypoth\`eses de la proposition~\ref{prop:compat} et comme~$E$ appartient \`a~$\mathcal{E'}(K ; q)$, on est dans le cas M0 ou dans le cas M1.
On remarque que d'apr\`es les propositions~\ref{prop:ramfrobhorspmult} et~\ref{prop:ramhorspbon}, le caract\`ere~$\mu$ est non ramifi\'e en toute place finie de~$K$ premi\`ere \`a~$p$.

Dans le cas M0, pour tout id\'eal premier~$\p$ de~$K$ au-dessus de~$p$,~$\ap$ est nul.
 Par d\'efinition de~$\ap$ (partie~\ref{ssec:inenp}), cela implique que caract\`ere~$\mu$ est  non ramifi\'e en toute place de~$K$ au-dessus de~$p$.
Ainsi, l'extension~$\Km$ de~$K$ trivialisant~$\mu$, qui est ab\'elienne, est non ramifi\'ee en toute place finie de~$K$~;
 le corps~$\Km$ est donc inclus dans le corps de classes de Hilbert de~$K$ et son degr\'e (sur~$\Q$) est inf\'erieur ou \'egal \`a~$d_K h_K$.
Le corps~$\Kl$ qui trivialise le caract\`ere~$\la$ est une extension de degr\'e divisant~$12$ de~$\Km$~;
son degr\'e (sur~$\Q$) est donc inf\'erieur ou \'egal \`a~$12 d_K h_K$.

Or, la courbe~$E$ poss\`ede un point d'ordre~$p$ d\'efini sur~$\Kl$.
D'apr\`es des travaux de Merel et Oesterl\'e mentionn\'es dans les introductions de~\cite{[Pa]} et~\cite{[Mer96]}, on a alors
$$
p \leq \left(   1 + 3^{ \frac{12 d_K h_K}{2} }   \right)^2 =  \left(   1 + 3^{ 6 d_K h_K }   \right)^2,
$$
ce qui contredit le choix de~$p$ strictement sup\'erieur \`a~$B(K ; q)$.

Dans le cas M1, pour tout id\'eal premier~$\p$ de~$K$ au-dessus de~$p$,~$\ap$ est \'egal \`a~$12$.
Ainsi, le caract\`ere~$\kip^{12}\mu^{-1}$ est non ramifi\'e en toute place de~$K$ au-dessus de~$p$, et par suite en toute place finie de~$K$.
On consid\`ere alors le quotient de~$E$ par le sous-groupe d'isog\'enie~$W$.
On obtient une courbe elliptique~$E'$ d\'efinie sur~$K$, isog\`ene \`a~$E$~sur~$K$~; le sous-groupe~$E[p]/W$ de~$E'$ est d\'efini sur~$K$ et d'ordre~$p$ et le caract\`ere d'isog\'enie associ\'e \`a ce sous-groupe est~$\kip\la^{-1}$.
La puissance douzi\`eme de~$\kip\la^{-1}$ \'etant non ramifi\'ee en toute place finie de~$K$, 
on applique le m\^eme raisonnement que pr\'ec\'edemment.
\end{proof}

\section{Forme du caract\`ere d'isog\'enie et homoth\'eties}\label{sec:carisohom}

\`A partir de maintenant et jusqu'\`a la fin du texte, on suppose \`a nouveau la repr\'esentation~$\phiEp$ r\'eductible.

\subsection{Une version effective du th\'eor\`eme de Chebotarev}\label{ssec:Cheboeff}

\begin{TheoCheboEff}
Il existe une constante absolue et effectivement calculable~$A$ ayant la propri\'et\'e suivante~:
soient~$M$ un corps de nombres,~$N$ une extension finie galoisienne de~$M$,~$\Delta_N$ le discriminant de~$N$,~$C$ une classe de conjugaison du groupe de Galois~$\Gal(N/M)$~;
alors il existe un id\'eal premier de~$M$, non ramifi\'e dans~$N$, dont la classe de conjugaison des 
frobenius dans l'extension~$N/M$ est la classe de conjugaison~$C$ et dont la norme dans l'extension~$M/\Q$ est un nombre premier rationnel inf\'erieur ou \'egal \`a~$2(\Delta_N)^A$. 
\end{TheoCheboEff}

\begin{definition}[Ensemble d'id\'eaux~$\J_K$]\label{def:JK}
On note~$\J_K$ l'ensemble des id\'eaux maximaux de~$K$ dont la norme dans l'extension~$K/\Q$ est un nombre premier rationnel totalement d\'ecompos\'e dans~$K$ et inf\'erieur ou \'egal \`a~$2(\Delta_{K})^{Ah_K}$.
\end{definition}

\begin{proposition}\label{prop:genclid}
Toute classe d'id\'eaux de~$K$ contient un id\'eal de~$\J_K$.
\end{proposition}

\begin{proof}
On note~$H_K$ le corps de classes de Hilbert de~$K$. D\'emontrer la proposition est \'equivalent \`a montrer que pour tout \'el\'ement~$\si$ du groupe de Galois de l'extension~$H_K/K$, il existe un id\'eal premier non nul~$\q$ de~$K$ appartenant \`a~$\J_K$ tel que l'\'el\'ement de Frobenius  associ\'e \`a~$\q$ dans~$\Gal(H_K/K)$ par l'application de r\'eciprocit\'e d'Artin  est \'egal \`a~$\si$.

Comme on a suppos\'e le corps~$K$ galoisien sur~$\Q$, le corps~$H_K$ est \'egalement une extension galoisienne de~$\Q$.
On va utiliser la version effective du th\'eor\`eme de Chebotarev pour les corps~$N = H_K$ et~$M = \Q$.
 On remarque que le discriminant~$\Delta_{H_K}$ de~$H_K$ est \'egal \`a~$\DK^{h_K}$  et que le groupe~$\Gal(H_K/K)$ est un sous-groupe (distingu\'e) du groupe~$\Gal(H_K/\Q)$. 

Soit~$\si$ un \'el\'ement de~$\Gal(H_K/K)$ et~$C$ la classe de conjugaison de~$\si$ dans~$\Gal(H_K/\Q)$.
D'apr\`es le th\'eor\`eme de Chebotarev effectif, il existe un nombre premier rationnel~$q$, non ramifi\'e dans~$H_K$ et inf\'erieur ou \'egal \`a~$2(\Delta_{H_K})^A$, 
 tel que la classe de conjugaison form\'ee par les frobenius de~$q$ dans l'extension~$H_K/\Q$ est \'egale \`a~$C$.
 Il existe donc un id\'eal~$\tilde{\q}$ de~$H_K$ au-dessus de~$q$ tel que le frobenius~$\Frob(\tilde{\q}/q)$ de~$\tilde{\q}$ dans l'extension~$H_K/\Q$ est \'egal  \`a~$\si$.

Soit~$\q$ l'id\'eal de~$K$ situ\'e au-dessous de~$\tilde{\q}$.
Alors la caract\'eristique de~$\q$ est le nombre premier rationnel~$q$.
 Par choix,~$q$ est inf\'erieur ou \'egal \`a~$2(\DK)^{Ah_K}$ et non ramifi\'e dans~$H_K$, donc dans $K$.
 Le frobenius~$\Frob(\q/q)$ de~$\q$ dans l'extension~$K/\Q$ est \'egal \`a la restriction \`a~$K$ du frobenius~$\Frob(\tilde{\q}/q)$ de~$\tilde{\q}$ dans~$H_K/\Q$.
  Or,~$\Frob(\tilde{\q}/q)$ est \'egal \`a~$\si$, qui est un \'el\'ement de~$\Gal(H_K/K)$.
   Ceci implique que~$\Frob(\q/q)$ est l'identit\'e de~$K$, donc que le degr\'e r\'esiduel de~$\q$ dans~$K/\Q$ est~$1$ et finalement que la norme de~$\q$ dans~$K / \Q$ est~$q$.
    Comme~$K$ est suppos\'e galoisien sur~$\Q$, on obtient que~$q$ est totalement d\'ecompos\'e dans~$K$ et ainsi que~$\q$ est dans l'ensemble~$\J_K$.

Enfin, comme~$\q$ est de degr\'e~$1$ dans~$K/\Q$, l'\'el\'ement de Frobenius~$\Frob(\tilde{\q}/\q)$ associ\'e \`a~$\q$ dans l'extension~$H_K/K$ v\'erifie
$$
\Frob(\tilde{\q}/\q) = \Frob(\tilde{\q}/q) = \si.
$$
\end{proof}

\subsection{Les deux formes possibles du caract\`ere d'isog\'enie}\label{ssec:2cariso}

Dans cette partie, on s'inspire de la d\'emonstration du th\'eor\`eme~1 (\S2) de~\cite{[Mom]} pour obtenir le th\'eor\`eme~II de l'introduction.

\begin{definition}[Borne $C_K$]
On pose
$$
C_K  = \max \left(  C\left( K , 2(\DK)^{Ah_K}  \right) ,   \left(  1 + 3^{6 d_K h_K}  \right)^2   \right).
$$
\end{definition}

Le nombre~$C_K$ ne d\'epend que du corps de nombres~$K$.

\begin{proposition}\label{prop:mmredJK}
On suppose que~$p$ est strictement sup\'erieur \`a~$C_K$.
Alors, en tout id\'eal de~$\J_K$, la courbe~$E$ a potentiellement bonne r\'eduction et tous les id\'eaux de~$\J_K$ appartiennent au m\^eme cas (BS, BO ou BO') de la proposition~\ref{prop:compat}.
\end{proposition}

\begin{proof}
Par construction, la borne~$C_K$ est sup\'erieure ou \'egale \`a~$ C\left( K , 2(\DK)^{Ah_K}  \right)$ et \`a~$B (K ; q)$ pour tout id\'eal~$\q$ de~$\J_K$ de caract\'eristique~$q$ .
Comme la repr\'esentation~$\phiEp$ est suppos\'e r\'eductible, les propositions~\ref{prop:critirrqq'} et~\ref{prop:critirrq} donnent que la courbe~$E$ n'appartient pas \`a la famille~$\mathcal{E}(K ; 2(\DK)^{Ah_K} )$ ni \`a la famille~$\mathcal{E'}(K ; q )$ pour tout premier rationnel~$q$ totalement d\'ecompos\'e dans~$K$ et inf\'erieur ou \'egal \`a~$2(\DK)^{Ah_K}$.

Ceci implique que~$E$ a m\^eme type de r\'eduction semi-stable (parmi les cinq types possibles de la proposition~\ref{prop:compat}) en tout id\'eal de~$\J_K$ et que ce type n'est pas multiplicatif.
\end{proof}

\subsubsection{Type supersingulier} \label{sssec:susec}

\begin{proposition}\label{prop:typeBS}
On suppose que~$p$ est strictement sup\'erieur \`a~$C_K$ et que tous les id\'eaux de~$\J_K$ sont de type supersingulier (BS dans la proposition~\ref{prop:compat}).
 Alors~:
\begin{enumerate}
	\item le nombre premier~$p$ est congru \`a $3$ modulo $4$~;
	\item en toute place de~$K$ au-dessus de~$p$, la courbe~$E$ a mauvaise r\'eduction additive et potentiellement bonne r\'eduction supersinguli\`ere~;
	\item on a~$\la^6 = \left( \kip\la^{-1} \right) ^6 =  \left( \kip^{\frac{p+1}{2}} \right)^3$. 
	\end{enumerate}
 \end{proposition}

\begin{proof}
Soit~$\p$ un id\'eal premier de~$K$ au-dessus de~$p$.
Comme tous les id\'eaux de~$\J_K$ sont de type BS, la proposition~\ref{prop:compat} donne que~$\ap$ est \'egal \`a $6$.
D'apr\`es les parties~\ref{sssec:eLbred} et~\ref{ssec:inenp} (notamment la  proposition~\ref{prop:inpbon}), ceci implique que~$p$ est congru \`a~$3$ modulo~$4$,~$E$ a potentiellement bonne r\'eduction supersinguli\`ere en~$\p$ et, comme~$\ep$ est \'egal \`a~$4$,~$E$ a r\'eduction additive en~$\p$.

D'apr\`es les propositions~\ref{prop:ramfrobhorspmult},~\ref{prop:ramhorspbon} et~\ref{prop:compat}, le caract\`ere~$\mu\kip^{-6}$ est ainsi non ramifi\'e en toute place finie de~$K$.
Il d\'efinit donc une extension ab\'elienne de~$K$ contenue dans son corps de classes de Hilbert~$H_K$.
Comme les classes des \'el\'ements de~$\J_K$ forment tout le groupe des classes d'id\'eaux de~$K$,
 le caract\`ere~$\mu \kip^{- 6}$  est enti\`erement d\'etermin\'e par sa valeur sur les \'el\'ements~$\sq$, lorsque~$\q$ d\'ecrit~$\J_K$.

Soit~$\q$ un id\'eal dans~$\J_K$~; alors d'apr\`es la proposition~\ref{prop:compat} (et avec les notations~\ref{not:poq}), on a~:
$$
\mu(\sq) = \bq^{12} \Mod \Poq = q^{6} \Mod p = (N\q)^{6} \Mod p = \kip(\sq)^{6}.
$$
On en d\'eduit que~$\mu$ est \'egal \`a~$\kip^{6}$ sur~$G_K$, ce qui implique~:
$$
\left(\kip \la^{-1} \right)^6 = \la^{12} \la^{ - 6} = \la^{ 6}
.
$$

Soit~$\psi$ le caract\`ere~$\kip \la^{-2}$ de~$G_K$ dans~$\Fpx$. On a
$
\psi^6 = \kip^6 \la^{-12} = \kip^6 \mu^{-1} \equiv 1
$
donc l'ordre de~$\psi$ divise~$6$.
On peut ainsi \'ecrire~$\psi$ comme produit d'un caract\`ere~$\psi_2$ d'ordre divisant~$2$ et d'un caract\`ere~$\psi_3$ d'ordre divisant~$3$.
Comme~$p$ est congru \`a~$3$ modulo~$4$ et que l'image de~$\psi_3$ est form\'ee de carr\'es de~$\Fpx$, on a~:
$$
\psi_2  =  \psi_2^{\frac{p-1}{2}} = \psi_2^{\frac{p-1}{2}} \psi_3^{\frac{p-1}{2}}  =   \psi^{\frac{p-1}{2}} =  \kip^{\frac{p-1}{2}} \la^{-(p-1)} = \kip^{\frac{p-1}{2}}.
$$
On en d\'eduit finalement~:
$$
\left( \kip\la^{-1} \right) ^6  =  \la^6  =  \left( \kip\psi^{-1} \right) ^3 =   \left( \kip\psi_2^{-1} \right) ^3 \psi_3^{-3} =  \left( \kip \kip^{\frac{p-1}{2}} \right) ^3.
$$
\end{proof}

\subsubsection{Type ordinaire} \label{sssec:ord}

\begin{proposition}
On suppose que~$p$ est strictement sup\'erieur \`a~$C_K$ et que tous les id\'eaux de~$\J_K$ sont de type BO ou que tous les id\'eaux de~$\J_K$ sont de type BO' (proposition~\ref{prop:compat}).
 Alors il existe un unique corps quadratique imaginaire~$L$ inclus dans~$K$ v\'erifiant~: pour tout id\'eal~$\q$ dans~$\J_K$,~$\Lq$ est \'egal \`a~$L$.
 \end{proposition}

\begin{proof}
Lorsque tous les id\'eaux dans~$\J_K$ sont de type BO, l'ensemble des \'el\'ements~$\tau$ de $G$ tels que~$\ato$ est \'egal \`a~$12$ est un sous-groupe d'indice~$2$ de~$G$, \'egal \`a~ $\Gal(K/\Lq)$ pour tout id\'eal~$\q$ de $\J_K$, $\Lq$ \'etant quadratique imaginaire et inclus dans $K$. (proposition~\ref{prop:compat}).
  Ainsi, le groupe $\Gal(K/\Lq)$, et par suite le corps $\Lq$, est ind\'ependant de l'id\'eal $\q$ dans $\J_K$. Ce corps quadratique commun fournit le corps~$L$ de la proposition.

Lorsque tous les id\'eaux dans~$\J_K$ sont de type BO', on applique le m\^eme raisonnement \`a l'ensemble des \'el\'ements~$\tau$ de $G$ tels que~$\ato$ est nul.
\end{proof}

\begin{proposition}
La norme dans l'extension~$K/L$ de tout id\'eal fractionnaire de~$K$ est un id\'eal fractionnaire principal de~$L$. Le corps de classes de Hilbert de~$L$ est contenu dans~$K$.
\end{proposition}

\begin{proof}
Comme toute classe d'id\'eaux de~$K$ contient un id\'eal de~$\J_K$ (proposition~\ref{prop:genclid}), 
la premi\`ere assertion est vraie si et seulement si elle l'est pour tous les id\'eaux de~$\J_K$.
Or pour tout id\'eal~$\q$ de~$\J_K$, la norme de~$\q$ dans l'extension~$K/L$ est la norme de~$\q$ dans l'extension~$K/ \Lq$, qui est un id\'eal principal engendr\'e par~$\bq$ (type BO) ou~$\bqb$ (type BO').

On montre ensuite que la premi\`ere assertion implique la deuxi\`eme.
Soit~$H_L$ le corps de classes de Hilbert de~$L$. L'extension $KH_L/K$ est galoisienne, ab\'elienne et la th\'eorie du corps de classes fournit le diagramme commutatif suivant~:
$$
\xymatrix{
\Gal(KH_L/K) \ar[r]^(0.375){\sim} \ar@{^{(}-_{>}}[d]_{restriction} & \A^{\times}_K/K^{\times}N_{KH_L/K}\left(\A^{\times}_{KH_L}\right) \ar[d]^(0.5){N_{K/L}} \\
\Gal(H_L/L) \ar[r]^(0.4){\sim} & \A^{\times}_L/L^{\times}N_{H_L/L}\left(\A^{\times}_{H_L}\right).\\ 
}
$$
D'apr\`es la premi\`ere assertion, la fl\`eche verticale de droite est d'image triviale. Comme la fl\`eche verticale de gauche est injective, les corps~$KH_L$ et~$K$ co\"incident, ce qui implique que~$H_L$ est inclus dans~$K$.
\end{proof}

\begin{proposition}\label{prop:pdecL}
Le nombre premier~$p$ est d\'ecompos\'e dans~$L$.
Il existe un unique id\'eal premier~$\pL$ de~$L$ au-dessus de~$p$ v\'erifiant~: pour tout id\'eal premier~$\p$ de~$K$ au-dessus de~$\pL$,~$\ap$ est \'egal \`a~$12$ et  pour tout id\'eal premier~$\p$ de~$K$ au-dessus de~$\overline{\pL}$, $\ap$ est \'egal \`a~$0$.
\end{proposition}

\begin{proof}
D'apr\`es la proposition~\ref{prop:compat},~$p$ est d\'ecompos\'e dans le corps~$\Lq$, \'egal \`a~$L$, pour tout~$\q$ dans~$\J_K$. 

On rappelle que les entiers~$\ato$ ont \'et\'e d\'efinis par~$\ato = a_{\tau^{-1}(\po)}$ pour un id\'eal~$\po$ de~$K$ au-dessus de~$p$ fix\'e (notations~\ref{not:po}).

Lorsque tous les id\'eaux de~$\J_K$ sont de type BO, on note~$\pL$ l'id\'eal de~$L$ situ\'e au-dessous de~$\po$ (pour tout id\'eal~$\q$ dans~$\J_K$,~$\pL$ est donc l'id\'eal~$\Poq$ des notations~\ref{not:poq}).
Soit~$\p$ un id\'eal premier de~$K$ au-dessus de~$\pL$.
Il existe un \'el\'ement~$\tau$ dans $\Gal (K / L)$ v\'erifiant~$\tau^{-1}(\po) = \p$.
Alors~$\ap$ est  \'egal \`a~$\ato$, lui-m\^eme \'egal \`a~$12$ car~$\tau$ est dans~$\Gal (K / L)$ qui co\"incide avec~$\Gal (K / \Lq)$.
R\'eciproquement, soit~$\p$ un id\'eal premier de~$K$ au-dessus de~$p$ tel que~$\ap$ est \'egal \`a~$12$.
Soit~$\tau$ dans~$G$ v\'erifiant $\tau^{-1}\po = \p$.
On a alors~$\ato = \ap = 12$ donc~$\tau$ est dans~$\Gal (K / \Lq)$, qui est \'egal \`a~$\Gal (K / L)$.
Cela implique~: 
$$
\p \cap \OL = \left(\tau^{-1}\po \right) \cap \OL = \tau^{-1}\left(\po \cap \OL \right) = \po \cap \OL = \pL.
$$
Ainsi on a~$\ap$ \'egal \`a~$12$ si et seulement si~$\p$ est au-dessus de~$\pL$ et~$\ap$ \'egal \`a~$0$ sinon~; dans ce deuxi\`eme cas,~$\p$ est au-dessus de~$\overline{\pL}$, qui est l'autre id\'eal de~$L$ au-dessus de~$p$.

Lorsque tous les id\'eaux de~$\J_K$ sont de type BO', on note~$\pL$ le conjugu\'e complexe de l'id\'eal de~$L$ situ\'e au-dessous de~$\po$
 (pour tout id\'eal~$\q$ dans~$\J_K$,~$\pL$ est donc le conjugu\'e complexe de l'id\'eal~$\Poq$ des notations~\ref{not:poq}).
  Soit~$\p$ un id\'eal premier de~$K$ au-dessus de~$\pL$.
  Il existe un \'el\'ement~$\tau$ dans~$G$ priv\'e de~$\Gal (K / L)  $ v\'erifiant  $\tau^{-1}(\po) = \p$.
  Alors~$\ap$ est \'egal \`a~$\ato$, donc \`a~$12$.
R\'eciproquement, soit~$\p$ un id\'eal premier de~$K$ au-dessus de~$p$ tel que~$\ap$ est \'egal \`a~$12$.
 Soit~$\tau$ dans~$G$ v\'erifiant~$\tau^{-1}(\po) = \p$.
  On a alors $\ato$ \'egal \`a~$\ap$ qui vaut~$12$, donc~$\tau$ est dans~$G$ priv\'e de~$\Gal (K / L)  $.
   Alors~$\tau$ et~$\tau^{-1}$ induisent sur~$L$ la conjugaison complexe et on obtient~: 
$$
\p \cap \OL = \left(\tau^{-1}(\po) \right) \cap \OL = \tau^{-1}\left(\po \cap \OL \right) = \overline{\po \cap \OL} = \pL.
$$
On a donc encore~$\ap$ \'egal \`a~$12$ si et seulement si~$\p$ est au-dessus de~$\pL$ et~$\ap$ \'egal \`a~$0$ sinon.
\end{proof}

\begin{proposition}\label{prop:recBO}
 Soient~$\al$ un \'el\'ement non nul de~$K$ premier \`a~$\pL$ et~$\prod_{\q \nmid \pL} \q^{\mathrm{val}_{\q}(\al)}$ la d\'ecomposition de l'id\'eal fractionnaire principal~$\al \OK$ en produit d'id\'eaux maximaux de~$K$. Alors on a~:
 $$
 \prod\limits_{\q \nmid \pL} \mu \left(\sq\right)^{\mathrm{val}_{\q}(\al)} = N_{K/L}\left(\al\right)^{12} \Mod\  \pL.
 $$ 
\end{proposition}

\begin{proof}
Le caract\`ere~$\mu$ est non ramifi\'e aux places finies de~$K$ premi\`eres \`a~$\pL$ et pour tout id\'eal premier de~$K$ au-dessus de $\pL$,~$\ap$ est \'egal \`a~$12$.
Par une d\'emonstration analogue \`a celle de la proposition~\ref{prop:recmu}, on obtient~:
$$
\prod\limits_{\q \nmid \pL} \mu \left(\sq\right)^{\mathrm{val}_{\q}(\al)} = \prod\limits_{\p | \pL} N_{\Kp/\Qp}\left(\iota_{\p}(\al)\right)^{12} \Mod  p.
$$
Le nombre premier~$p$ \'etant d\'ecompos\'e dans~$L$ (proposition~\ref{prop:pdecL}), le compl\'et\'e~$L_{\pL}$ de~$L$ en~$\pL$ co\"incide avec~$\Qp$ et on a~:
 $$
 \begin{array}{r c l}
 \prod\limits_{\p | \pL} N_{\Kp/\Qp}\left(\iota_{\p}(\al)\right)^{12} \Mod\ p & = & \left( \prod\limits_{\p | \pL} N_{\Kp/L_{\pL}}\left(\iota_{\p}(\al)\right)\right)^{12} \Mod  p \\
  & = &  N_{K/L}\left( \al \right) ^{12} \Mod \pL .\\
\end{array}
$$
\end{proof}

\begin{proposition}\label{prop:typeBO}
Soient~$\q$ un id\'eal maximal de~$K$ premier \`a~$\pL$ et~$\alq$ dans~$L$ un g\'en\'erateur de~$N_{K/L}(\q)$. Alors on a~:
$$
\mu(\sq) = \alq^{12} \Mod \pL.
$$
\end{proposition}

\begin{proof}
Il existe un id\'eal~$\q'$ dans $\J_K$ et un \'el\'ement~$\al$ non nul de~$K$ v\'erifiant~:
$$
\q = \q' \cdot \left(\al \OK \right).
$$

D'apr\`es la d\'emonstration de la proposition~\ref{prop:pdecL}~:
\begin{itemize}
\item[$\bullet$] si tous les id\'eaux de~$\J_K$ sont de type BO, alors~$\pL$ est \'egal \`a l'id\'eal~$\Poqpr$, la norme de~$\q'$ dans l'extension~$K/L$ est~$\bqpr\OL$ et on a (notations \ref{not:poq} et proposition~\ref{prop:compat})
$$
\mu(\sqpr) = \bqpr^{12} \Mod \Poqpr = \bqpr^{12} \Mod \pL ;
$$
\item[$\bullet$] si tous les id\'eaux de~$\J_K$ sont de type BO' alors~$\pL$ est le conjugu\'e complexe de l'id\'eal~$\Poqpr$, la norme de~$\q'$ dans~$K/L$ est~$\bqprb \OL$ et on a
$$
\mu(\sqpr) = \bqpr^{12} \Mod \Poqpr = \bqpr^{12} \Mod \overline{\pL} = \bqprb^{12} \Mod \pL.
$$

\end{itemize}
Dans les deux cas, il existe un \'el\'ement~$\al_{\q'}$ de~$\OL$ qui engendre l'id\'eal~ $N_{K/L}(\q')$ et v\'erifie~:
$$
\mu(\sqpr) = \al_{\q'}^{12} \Mod \pL.
$$

Soit~$\alq$ dans~$\OL$ un g\'en\'erateur de l'id\'eal~$N_{K/L}(\q)$. On a~:
$$
N_{K/L}(\q) = N_{K/L}(\q') \times \left(N_{K/L}(\al)\OL \right),
$$
ce qui implique que les \'el\'ements~$\alq$ et~$ \al_{\q'}N_{K/L}(\al)$ de~$L$ diff\`erent d'une unit\'e de~$L$. 
Comme~$L$ est un corps quadratique imaginaire, cette unit\'e est une racine douzi\`eme de l'unit\'e~; on a donc
$$
\alq^{12} = \left(\al_{\q'}N_{K/L}(\al)\right)^{12}.
$$
Enfin, on a d'apr\`es la proposition~\ref{prop:recBO}
 $$
  \mu(\sq)  = \mu(\sqpr) \times \left(N_{K/L}\left(\al\right)^{12} \Mod\  \pL \right)
  $$
 d'o\`u
 $$
 \mu(\sq)  = \left(\al_{\q'}^{12} \Mod \pL \right) \times \left(N_{K/L}\left(\al\right)^{12} \Mod\  \pL \right) = \alq^{12} \Mod \pL.
 $$
\end{proof}

\subsection{Homoth\'eties dans l'image de la repr\'esentation~$\phiEp$}

On remarque qu'on a, pour tout couple~$( \alpha , \beta)$ dans~$\Fpx \times \Fp $, l'\'egalit\'e suivante dans $\mathrm{M}_2(\Fp)$~:
$$
\begin{pmatrix}
\al & \beta\\
0 & \al
\end{pmatrix}^{p}
 =
\begin{pmatrix}
\al^{p} & p\al^{p-1}\beta\\
0 & \al^{p}
\end{pmatrix}
 =
\begin{pmatrix}
\al^{p} & 0\\
0 & \al^{p}
\end{pmatrix}
 =
\begin{pmatrix}
\al & 0\\
0 & \al
\end{pmatrix}.
$$
Pour montrer que l'image de~$\phiEp$ contient une homoth\'etie de rapport~$x$ ($x$ appartenant \`a~$\Fpx$), il suffit donc de montrer que l'image de~$\phiEp$ contient un \'el\'ement dont la diagonale est~$(x , x)$.

\subsubsection{Homoth\'eties pour le type supersingulier}

\begin{proposition}\label{prop:homBS}
 On suppose que~$p$ est strictement sup\'erieur \`a~$C_K$ et qu'on est dans le cas supersingulier (partie~\ref{sssec:susec}).
  Alors l'image de~$\phiEp$ contient les carr\'es des homoth\'eties. 
 \end{proposition}

\begin{proof}
Soit~$\p$ un id\'eal premier de~$K$ au-dessus de~$p$.
Comme~$\ap$ est \'egal \`a~$6$ (proposition~\ref{prop:compat}), la proposition~\ref{prop:inpbon} donne que le caract\`ere~$\la^4$ restreint au sous-groupe d'inertie~$\Ip$ est \'egal au carr\'e du caract\`ere cyclotomique.
Le deuxi\`eme caract\`ere diagonal de~$\phiEp$ est~$\kip\la^{-1}$ et v\'erifie alors \'egalement en restriction \`a~$\Ip$~:
$$
\left( \kip\la^{-1} \right)^4 = \kip^4 \la^{-4} = \kip^4 \kip^{-2 } = \kip^{2 }.
$$
Comme on a suppos\'e~$p$ non ramifi\'e dans~$K$, le caract\`ere cyclotomique restreint \`a~$\Ip$ est surjectif dans~$\Fpx$.
Ainsi, pour tout~$x$ dans~$\Fpx$, l'image de~$\phiEp$ contient un \'el\'ement de diagonale~$(x^2 , x^2)$ et la puissance $p$-i\`eme de cet \'el\'ement qui est une homoth\'etie de rapport~$x^2$.
\end{proof}

\subsubsection{Homoth\'eties pour le type ordinaire}

\begin{proposition}\label{prop:homBO}
On suppose que~$p$ est strictement sup\'erieur \`a~$C_K$ et qu'on est dans le cas ordinaire (partie~\ref{sssec:ord}).
 Alors 
   l'image  de~$\phiEp$ contient les puissances douzi\`emes des homoth\'eties. 
\end{proposition}

\begin{proof}
Soit $x$ un \'el\'ement de $\Fpx$.

Soit~$\p$ un id\'eal premier de~$K$ au-dessus de~$\pL$. 
D'apr\`es la proposition~\ref{prop:pdecL},~$\ap$ est \'egal \`a~$12$ donc la restriction \`a~$\Ip$ de la diagonale de~$\phiEp$, \'elev\'ee \`a la puissance~$12$, vaut~$( \mu , \kip^{12} \mu^{-1} ) =  (\kip^{12} , 1 ) $.
Comme~$\kip$ est surjectif de~$\Ip$ dans~$\Fpx$, il existe dans~$G_K$ un \'el\'ement~$\si$ tel que la diagonale de~$\phiEp(\si)$ est~$(x^{12} , 1)$.

Soit~$\p'$ un id\'eal premier de~$K$ au-dessus de~$\overline{\pL}$. 
D'apr\`es la proposition~\ref{prop:pdecL},~$a_{\p'}$ est nul donc la restriction \`a~$I_{\p'}$ de la diagonale de~$\phiEp$, \'elev\'ee \`a la puissance~$12$, vaut~$( \mu , \kip^{12} \mu^{-1} ) =  ( 1 , \kip^{12} ) $.
Comme~$\kip$ est surjectif de~$I_{\p'}$ dans~$\Fpx$, il existe dans~$G_K$ un \'el\'ement~$\si'$ tel que la diagonale de~$\phiEp(\si')$ est~$(1 , x^{12})$.

Alors l'image par~$\phiEp$ du produit~$\si \si'$ de~$G_K$  a pour diagonale~$(x^{12} ,x^{12})$~; sa puissance~$p$-i\`eme est donc une homoth\'etie de rapport~$x^{12}$ contenue dans l'image de~$\phiEp$.
\end{proof}

\nocite{}
\shorthandoff{:}
\bibliographystyle{plain}
\bibliography{Cici.bib}

\begin{thebibliography}{10}

\bibitem{[Bil]}
Nicolas Billerey.
\newblock Crit\`eres d'irr\'eductibilit\'e pour les repr\'esentations des
  courbes elliptiques.
\newblock {\em Int. J. Number Theory}, 7(4):1001--1032, 2011.

\bibitem{[BuGy]}
Yann Bugeaud and K{\'a}lm{\'a}n Gy{\H{o}}ry.
\newblock Bounds for the solutions of unit equations.
\newblock {\em Acta Arith.}, 74(1):67--80, 1996.

\bibitem{[Dav08]}
Agn\`es David.
\newblock {\em Caract\`ere d'isog\'enie et borne uniforme pour les
  homoth\'eties}.
\newblock PhD thesis, IRMA, Strasbourg, 2008.

\bibitem{[Hudig]}
Agn\`es David.
\newblock Borne uniforme pour les homoth\'eties dans l'image de {G}alois
  associ\'ee aux courbes elliptiques.
\newblock {\em J. Number Theory}, 131(11):2175--2191, 2011.
\newblock Disponible en ligne arXiv:1007.4725v1.

\bibitem{[Krau]}
Alain Kraus.
\newblock Courbes elliptiques semi-stables sur les corps de nombres.
\newblock {\em Int. J. Number Theory}, 3(4):611--633, 2007.

\bibitem{[LMO]}
J.~C. Lagarias, H.~L. Montgomery, and A.~M. Odlyzko.
\newblock A bound for the least prime ideal in the {C}hebotarev density
  theorem.
\newblock {\em Invent. Math.}, 54(3):271--296, 1979.

\bibitem{[LV]}
Eric Larson and Dmitry Vaintrob.
\newblock Determinants of subquotients of {G}alois representations associated
  to abelian varieties.
\newblock {\em preprint}, 2011.
\newblock Disponible en ligne arXiv:1110.0255v2.

\bibitem{[Maz]}
B.~Mazur.
\newblock Rational isogenies of prime degree (with an appendix by {D}.
  {G}oldfeld).
\newblock {\em Invent. Math.}, 44(2):129--162, 1978.

\bibitem{[Mer96]}
Lo{\"{\i}}c Merel.
\newblock Bornes pour la torsion des courbes elliptiques sur les corps de
  nombres.
\newblock {\em Invent. Math.}, 124(1-3):437--449, 1996.

\bibitem{[Mom]}
Fumiyuki Momose.
\newblock Isogenies of prime degree over number fields.
\newblock {\em Compositio Math.}, 97(3):329--348, 1995.

\bibitem{[Pa]}
Pierre Parent.
\newblock Bornes effectives pour la torsion des courbes elliptiques sur les
  corps de nombres.
\newblock {\em J. Reine Angew. Math.}, 506:85--116, 1999.

\bibitem{[Ray]}
Michel Raynaud.
\newblock Sch\'emas en groupes de type {$(p,\dots, p)$}.
\newblock {\em Bull. Soc. Math. France}, 102:241--280, 1974.

\bibitem{[Se]}
Jean-Pierre Serre.
\newblock Propri\'et\'es galoisiennes des points d'ordre fini des courbes
  elliptiques.
\newblock {\em Invent. Math.}, 15(4):259--331, 1972.

\bibitem{[SeTa]}
Jean-Pierre Serre and John Tate.
\newblock Good reduction of abelian varieties.
\newblock {\em Ann. of Math. (2)}, 88:492--517, 1968.

\bibitem{[Sil]}
Joseph~H. Silverman.
\newblock {\em The arithmetic of elliptic curves}, volume 106 of {\em Graduate
  Texts in Mathematics}.
\newblock Springer-Verlag, New York, 1992.
\newblock Corrected reprint of the 1986 original.

\end{thebibliography}

\end{document}